\documentclass[a4j,12pt,leqno]{amsart}
\usepackage{comment}
\usepackage{url}
\usepackage{amsmath,amssymb}
\usepackage[all,cmtip]{xy}
\usepackage{shuffle}
\usepackage{mathrsfs}
\usepackage{a4wide}
\usepackage{color}
\makeatletter
\renewcommand{\section}{\@startsection
	{section}{1}{0mm}{5mm}{2mm}{\raggedright\bfseries}}
\@addtoreset{equation}{section}
\makeatother

\theoremstyle{plain} 
\newtheorem{theorem}{\indent\sc Theorem}[section] 
\newtheorem{lemma}[theorem]{\indent\sc Lemma}
\newtheorem{corollary}[theorem]{\indent\sc Corollary}
\newtheorem{proposition}[theorem]{\indent\sc Proposition}

\theoremstyle{definition} 
\newtheorem{definition}[theorem]{\indent\sc Definition}
\newtheorem{remark}[theorem]{\indent\sc Remark}
\newtheorem{example}[theorem]{\indent\sc Example}

\newcommand\rmB{{\mathrm B}}

\newcommand\rmH{{\mathrm H}}

\newcommand\rmO{{\mathrm O}}

\newcommand\rmR{{\mathrm R}}

\newcommand\rmf{{\mathrm f}}

\newcommand\bfA{{\mathbf A}}

\newcommand\bfC{{\mathbf C}}

\newcommand\bfP{{\mathbf P}}
\newcommand\bfQ{{\mathbf Q}}
\newcommand\bfR{{\mathbf R}}

\newcommand\bfZ{{\mathbf Z}}

\newcommand\bbD{{\mathbb D}}

\newcommand\calA{{\mathcal A}}
\newcommand\calB{{\mathcal B}}
\newcommand\calC{{\mathcal C}}

\newcommand\calE{{\mathcal E}}
\newcommand\calF{{\mathcal F}}

\newcommand\calH{{\mathcal H}}

\newcommand\calK{{\mathcal K}}
\newcommand\calL{{\mathcal L}}

\newcommand\calO{{\mathcal O}}

\newcommand\calR{{\mathcal R}}

\newcommand\calV{{\mathcal V}}
\newcommand\calW{{\mathcal W}}

\newcommand\frakH{{\mathfrak H}}

\newcommand\bo{\boldsymbol o}

\def\B+{\mathop{{\mathbf B}^+_{\mathbf{dR}}}\nolimits}

\def\D+{\mathop{D^0_{\mathbf{dR}}}\nolimits}

\def\RepCrys_F{\mathop{\mathbf{Rep}_{\Q_\ell}^{\mathbf{crys}}(G_F)}\nolimits}

\newcommand\End{\mathrm{End}}

\newcommand\Hom{\mathrm{Hom}}
\newcommand\Mor{\mathrm{Mor}}
\newcommand\Ker{\mathrm{Ker}}

\newcommand\Ext{\mathrm{Ext}}

\def\Vec{\mathop{\mathrm{Vec}}\nolimits}

\newcommand\GL{\mathbf{GL}}
\newcommand\SL{\mathbf{SL}}

\newcommand\Coker{\mathrm{Coker}}

\newcommand\dR{{\mathrm{dR}}}

\newcommand\reg{{\mathrm{reg}}}

\newcommand\Gal{\mathrm{Gal}}

\def\IRep_F{\mathop{\mathrm{Ind}\mathbf{Rep}_{\Q_p}^{\mathbf{crys}}(G_F)}\nolimits}

\newcommand\spec{\mathrm{Spec}}

\newcommand\Obj{{\mathrm{Obj}}}

\def\o+{{\oplus}}

\def\bo+{{\bigoplus}}

\newcommand\op{{\mathrm{op}}}

\newcommand\an{{\mathrm{an}}}

\newcommand\la{\langle}
\newcommand\ra{\rangle}
\newcommand\pr{{\mathrm{pr}}}

\newcommand\modulo{{\ \mathrm{mod}\ }}
\newcommand\Cusp{{\mathrm{Cusp}}}

\def\b0{{\boldsymbol 0}}
\def\top{\mathrm{top}}

\newcommand\ontomap{\twoheadrightarrow}
\newcommand\intomap{\hookrightarrow}
\newcommand\isom{\xrightarrow{\ \sim\ }}
\newcommand{\mapfor}[2]{\quad \text{for } #1 \in #2}
\begin{document}
\title[The infinite Frobenius action on de Rham fundamental groups]
{On the infinite Frobenius action on de Rham fundamental groups of affine curves}
\author{Kenji Sakugawa}

\address{Kenji Sakugawa: \\
	Faculty of Education, Shinshu University, 6-Ro, Nishi-nagano, Nagano 380-8544, Japan.}
\email{sakugawa\_kenji@shinshu-u.ac.jp}
\maketitle

\begin{abstract}We study the action of the infinite Frobenius on the de Rham fundamental groups
	of affine curves defined over $\bfR$.
	As an application, we compute extension classes of real mixed Hodge structures associated with the motivic fundamental groups of affine curves.
	In the case of modular curves, we relate our computation to special values of Rankin-Selberg L-functions, and show that the associated extensions of mixed Hodge structures are non-split.
	We compute local zeta integrals both at good primes and, in certain cases, at bad primes.
\end{abstract}
%\tableofcontents

\section{Introduction}\label{intro}
\begin{comment}
The topological fundamental group $\pi_1(X)$ of an algebraic variety $X$ over $\bfQ$ carries more information than just the topological structure of $X(\bfC)$. For example, its profinite completion $\widehat{\pi}_1(X)$ is equipped with a continuous outer action of $\Gal(\overline{\bfQ}/\bfQ)$ which is faithful when $X$ is a hyperbolic curve. Though this Galois representation is very interesting, to extract arithmetic information from $\widehat{\pi}_1(X)$ is not easy. An idea by Ihara to analyze this Galois action is to take a mild abelianization
$\bfZ_p[\![\widehat{\pi}_1(X)]\!]/J^{N+1}$ of $\widehat{\pi}_1(X)$, where $J$ is the augmentation ideal of the (pro-$p$) Iwasawa algebra $\bfZ_p[\![\widehat{\pi}_1(X)]\!]$.
In~\cite{Ihara86}, Ihara found explicit relations between the Galois representation above, Jacobi sums, and Kubota-Leopold's $p$-adic L-functions when $X=\bfP^1\setminus\{0,1,\infty\}$.
The natural question arising from Ihara's studies is:
What about other hyperbolic curves, and is there any Hodge theoretic analogue?

Hodge theory of fundamental groups has a different origin of the explained above.
\end{comment}
Let $X$ be an algebraic variety over a subfield of $\bfC$, and let $\pi_1(X)$ denote its topological fundamental group.
Let $\bfZ[\pi_1(X)]$ be the group ring of $\pi_1(X)$, and let $J$ be its augmentation ideal.
By the mid-1980s, Morgan (\cite{Morgan}) and  Hain (\cite{Hain87b}) constructed mixed Hodge structures on the truncated
group ring $\bfZ[\pi_1(X)]/J^{n+1}$ of the fundamental group $\pi_1(X)$, using Sullivan's minimal models and Chen's iterated integrals, respectively.
The aim of this paper is to compute extension classes of Hain's mixed Hodge structure
on truncated group rings of fundamental groups of affine curves in the case $n=2$.

Let $Y$ be an affine curve over $\bfR$, and let $b$ be an $\bfR$-rational base point.
Let $\pi_1(Y(\bfC),b)$ be the topological fundamental group of $Y(\bfC)$, and
let $I_b$ denote the augmentation ideal of $\bfZ[\pi_1(Y(\bfC),b)]$.
There is a natural isomorphism
\begin{equation}
	\label{eq0}
	I_b^{n}/I_b^{n+1}\cong\rmH_1(Y(\bfC),\bfZ)^{\otimes n}
\end{equation}
of abelian groups.
For each positive integer $n$, Hain constructed a natural mixed Hodge structure on $I_b/I_b^{n+1}$ for a general variety $Y$, as is precisely described in~\cite{HainZucker}.
This mixed Hodge structure has the property that the natural homomorphism
$I_b/I_b^{n+1}\ontomap I_b/I_b^n$ is a morphism of mixed Hodge structures,
and that the natural isomorphism (\ref{eq0}) is an isomorphism of mixed Hodge structures.
In particular, when $n=2$, we obtain a short exact sequence
\begin{equation}
	\label{eq-1}
	0\to \rmH_1(Y(\bfC),\bfZ)^{\otimes 2}\to I_b/I_b^{\otimes 3}\to \rmH_1(Y(\bfC),\bfZ)\to 0
\end{equation}
of mixed Hodge structures.
It is well known that this mixed Hodge structure carries the so-called infinite
Frobenius $\varphi_{\infty}$, induced by the complex conjugation on $Y(\bfC)$.
We write $\mathcal{MH}_A^+$ as the category of $A$-mixed Hodge structures with infinite
Frobenius. The main results of this paper are Theorem~\ref{thm3-0}, Theorem~\ref{thm3-1}, Theorem~\ref{thm3-2}, and Theorem~\ref{thm3-3}, which compute the matrix coefficients of $\varphi_{\infty}$ with respect to the \emph{de Rham real structure} in terms of inner products of differential forms on $Y(\bfC)$. That is, they give an explicit description of the action of $\varphi_{\infty}$ on the extension (\ref{eq0}). Those results determine the extension class (\ref{eq-1}) in $\mathcal{MH}_{\bfR}^+$ in a partial but explicit way (Theorem~\ref{thm5-1} and Theorem~\ref{thm5-2}).

In Section 6, we specialize to the case where $Y$ is a modular curve. We relate inner products appearing in Theorem~\ref{thm4-2} and Theorem~\ref{thm4-3} to special values of Rankin-Selberg L-functions.
For simplicity, we illustrate our results in this introduction by focusing on the case $Y=Y_0(p)$, where $p$ is a prime number.
In this case, the extension class (\ref{eq-1}) in the category $\mathcal{MH}_{\bfR}^+$, that is denoted by $[I_b/I_b^3]$, is partially determined by the image of the homomorphism
\[
\reg_{\bfR}\colon \Ext^1_{\mathcal{MH}_{\bfZ}^+}(\rmH_1(Y(\bfC),\bfZ)^{\otimes 2},\rmH_1(Y(\bfC),\bfZ))\to \bigoplus_{f,g}\bfR e_{f,g},
\]
which is defined in Subsection~\ref{MHSwithfrob}.
Here $f$ and $g$ range over normalized cuspidal Hecke eigenforms of weight two and level $\Gamma_0(p)$.
The unnormalized Petersson inner product of $f$ and $g$ is denoted by $(f,g)$,
and $\epsilon_f$ denotes the sign of the functional equation of $f$.
Let $L(s,\pi_f\times \pi_g)$ be a Rankin-Selberg L-function associated with the automorphic representations $\pi_f$ and $\pi_g$ defined in (\ref{eq40}). This L-function is a holomorphic function on $\bfC$ if $f\neq g$ and has a simple pole at $s=1$ if $f=g$.
Note that our definition differs slightly from that of Jacquet in~\cite{Ja}.
See Remark~\ref{rem6.0}.
\begin{theorem}\label{thm1}When $b$ lies over an $\bfR$-rational point of $Y_0(p)$,
	then we have that
	\begin{multline}
		\reg_{\bfR}([I_b/I_b^3])=-2\pi\sqrt{-1}\sum_{f\neq g,\ \epsilon_f\epsilon_g=-1}\frac{L(1,\pi_f\times \pi_g)}{\zeta^{(p)}(2)(4\pi)^2(f,f)(g,g)}e_{f,g}\\
		-2\pi\sqrt{-1}\sum_f \left(\frac{\mathrm{Res}_{s=1}(L(s,\pi_f\times \pi_f))\log p}{\zeta^{(p)}(2)(4\pi)^2(f,f)^2}-\frac{\calE_0(b,1)}{4\pi(f,f)}\right)e_{f,f}.
	\end{multline}
	Here, $\zeta^{(p)}(s)=\prod_{\ell\neq p}(1-\ell^{-s})^{-1}$, and $\calE_0(\tau,s)$ denotes the Eisenstein series of weight zero defined in Subsection~\ref{eis}.
	
	If $b$ is the standard tangential base point $\partial/\partial q$ $($\cite[Subsection 4.1]{Brown17}$)$,
	then the following equation holds$:$
	\begin{multline}
		\reg_{\bfR}([I_b/I_b^3])=-2\pi\sqrt{-1}\sum_{f\neq g,\ \epsilon_f\epsilon_g=-1}\frac{L(1,\pi_f\times \pi_g)}{\zeta^{(p)}(2)(4\pi)^2(f,f)(g,g)}e_{f,g}\\
		-2\pi\sqrt{-1}\sum_f \frac{\mathrm{Res}_{s=1}(L(s,\pi_f\times \pi_f))\log p}{\zeta^{(p)}(2)(4\pi)^2(f,f)^2}e_{f,f}.
	\end{multline}
	In particular, the exact sequence (\ref{eq-1}) does \emph{not} split in $\mathcal{MH}_{\bfR}^+$ when $b=\partial/\partial q$.
\end{theorem}
We would like to emphasize that in order to prove the non-triviality of the extension class defined by (\ref{eq-1}), we need to compute the local zeta integrals at not only good primes but also at \emph{bad primes}.

Note that the L-function $L(s,\pi_f\times \pi_g)$ is \emph{non-critical} at $s=1$.
In~\cite{DRS15}, Darmon, Rotger, and Sols showed that a part of the extension class (\ref{eq-1}) in $\mathcal{MH}_{\bfQ}^+$
can be described in terms of central critical values of the triple product L-functions, in the case where $Y$ is a modular curve or a Shimura curve. Theorem~\ref{thm1} may thus be viewed as complementary to their result.

Our method for computing the infinite Frobenius action extends naturally to the \emph{relatively unipotent} case.
In~\cite{Brown17}, Brown studied mixed modular motives (MMM), a subcategory of ``mixed motives'' generated by the relative pro-unipotent completion of $\SL_2(\bfZ)$. The techniques developed in this paper appear to be applicable to the construction of Rankin--Selberg-type generators for the fundamental Lie algebra of MMM.

The organization of this paper is as follows.
In Section 2, we recall the mixed Hodge structures on $I_b/I_b^{n+1}$ constructed by Hain (\cite{Hain87b}) and Hain--Zucker (\cite{HainZucker}).
In Section 3, we give a generalization of the Bloch--Wigner function (Definition~\ref{defofBWfct}), which is, by definition, a single-valued function on $Y(\bfC)$. Although its appearance differs from the classical Bloch--Wigner function, we show that it indeed generalizes the classical one. This construction is inspired by ideas of Brown~\cite[18.4]{Brown17}. The main results of this paper are presented in Section 4, where we compute the matrix coefficients of the infinite Frobenius on $I_b/I_b^{3}$
in terms of inner products of smooth one-forms on $Y(\bfC)$. Brown obtained a similar result in
\cite[Section 9]{Brown17} by using a generalization of Haberland formula (cf.\ ~\cite{PP}), but our method differs from his.
In Section 5, we give regulator formulas based on the calculations in Section 4.
In Section 6, we focus on the modular curve $Y_0(N)$.
We first recall how the inner products in Sections 4 and 5 can be related to certain zeta integrals of automorphic forms on $\GL_2(\bfA)$.
Then, we compute such local zeta integrals not only at good primes but also at bad primes
in certain cases (Theorem~\ref{thm5-3}).
\subsection*{Notation}
For a field $k$ and a $k$-vector space $V$, the symbol $V^\vee$ denotes
the $k$-dual of $V$. The canonical pairing between $V$ and $V^\vee$
is denoted by
\[
\la\ ,\ \ra\colon V\times V^\vee\to k.
\]
The category of finite-dimensional $k$-vector spaces is denoted by $\Vec_k$.

Let $M$ be a topological space and let $\gamma,\delta\colon [0,1]\to M$
be a composable paths. In this paper, we use algebraists' convention of path compositions, namely,
$\gamma\delta$ is defined to be the path first through $\delta$ and then through $\gamma$.

Let $\epsilon$ and $\epsilon'$ be $+$ or $-$. Then, the sign $\epsilon \epsilon'\in\{+,-\}$
is defined by the equation $++=--=+,\ +-=-+=-$. 
%%%%%%%%%%%%%%%%%%%%%
%%%%%%%%%%%%%%%%%%%%%
%%%%%%%%%%%%%%%%%%%
\section{Mixed Hodge structures on fundamental groups of algebraic varieties}\label{MHS}
Let $F$ be a subfield of $\bfC$, and let $Y$ be a smooth, geometrically connected affine curve over $F$.
We denote by $\pi^{\top}_1(Y,y)$ the topological fundamental group of $Y(\bfC)$
with the base point $y\in Y(\bfC)$.
For simplicity, we assume that $y$ lies over an $F$-rational point of $Y$.
Let $I$ denote the augmentation ideal of the group ring $\bfZ[\pi_1^{\top}(Y,y)]$. We refer to the quotients $\bfZ[\pi_1^{\top}(Y,y)]/I^{n+1}$ as the
\emph{truncated group rings of $\pi_1^{\top}(Y,y)$}.
Morgan and Hain constructed
a natural mixed Hodge structure on the (limit of) truncated group ring
(\cite{Morgan},~\cite{Hain87b}).
In this section, we recall Hain's construction of mixed Hodge structures on the truncated group rings of $ \pi_1^{\top}(Y,y)$.

Let $\calC^{\rmB}(Y)$ and $\calC^{\dR}(Y)$ denote the categories of unipotent local systems on $Y(\bfC)$
and unipotent flat connections on $Y$, respectively.
By the definition, they are unipotent Tannakian categories in the sense of the Appendix.
Note that, due to the finiteness of the Betti and de Rham cohomology groups,
they also satisfy the condition in Proposition~\ref{propa1}.
For each symbols $\bullet=\dR,\ \rmB$, we denote by  $\omega_y^{\bullet}$ the associated fiber functor from $\calC^{\bullet}(Y)$ to the base point $y$.

Let $(\calL_N,s_N)$ be the $N$th layer of the universal pro-$\omega_y^{\rmB}$-marked object of $\calC^{\rmB}(Y)$ (see Definition~\ref{dfna1}).
By definition, we have a short exact sequence
\[
0\to \rmH^1(Y(\bfC),\calL_{N-1}^\vee)^\vee\to \calL_N\to \calL_{N-1}\to 0
\]
of $\bfQ$-local systems where the connecting homomorphism
\[
\rmH^1(Y(\bfC),\calL_{N-1}^\vee)\to \Ext^1_{\calC^{\rmB}(Y)}(\calL_{N-1},\bfQ)\cong \rmH^1(Y(\bfC),\calL_{N-1}^\vee)
\]
is the identity map.
Here, $\rmH^1(Y(\bfC),\calL_{N-1}^\vee)^\vee$ denotes the constant local system associated with this $\bfQ$-vector space by abuse of notation.
Since $Y$ is an affine curve, we have that $\rmH^2(Y(\bfC), \calF)=0$ for any constructible sheaf $\calF$ on $Y(\bfC)$ (\cite[Theorem 2.5.23]{HuM17}). Therefore, we have the following natural isomorphism:
\[
\rmH^1(Y(\bfC),\calL_{N-1}^\vee)\xrightarrow{\ \sim\ }\rmH^1(Y(\bfC),\bfQ)^{\otimes (N-1)}.
\]
Note that the category $\calC^{\rmB}(Y)$ is equivalent to the category of unipotent representations
of $\pi_1^{\mathrm{top}}(Y,y)$ on finite-dimensional $\bfQ$-vector spaces.
Hence, $(\calL_N,s_N)$ is canonically isomorphic to the marked local system associated with $(\bfQ[\pi_1^{\mathrm{top}}(Y,y)]/I^{N},1)$.
Consequently, for any (possibly tangential) base point $z$, we have that
\[
\calL_{N,z}\cong \bfQ[\pi_1^{\mathrm{top}}(Y;y,z)]/\bfQ[\pi_1^{\mathrm{top}}(Y;y,z)]I^{N}.
\]
Here, $\pi_1^{\mathrm{top}}(Y;y,z)$ is the set of homotopy classes of paths from $y$ to $z$.

Let $((\calV_N,\Delta_N),t_N)$ be the $N$th layer of a universal pro-$\omega_y^{\dR}$-marked object of $\calC^{\dR}(Y)$. Similarly to the Betti case, we have a short exact sequence
\[
0\to \rmH^1_{\dR}(Y,\calV_{N-1}^\vee)^\vee\to \calV_N\to \calV_{N-1}\to 0
\]
of flat connections on $Y$ such that the induced connecting homomorphism is the identity map.
As in the Betti case, we have
\[
\rmH^1_{\dR}(Y,\calV_{N-1}^\vee)\xrightarrow{\ \sim\ }\rmH^1_{\dR}(Y/F)^{\otimes(N-1)}
\]
since $Y$ is an affine curve.

There is a natural functor
\begin{equation}
	\label{unipRH}
	\calC^{\dR}(Y)\otimes\bfC\to \calC^{\rmB}(Y)\otimes\bfC;\quad (\calV,\nabla)\mapsto \calV_{\an}^{\nabla=0}
\end{equation}
between Tannakian categories, where $\calV_{\an}$ is the analytification of $\calV$.
The unipotent version of the Riemann--Hilbert correspondence
(\cite{Hain87}) asserts that the functor (\ref{unipRH}) is an equivalence of $\bfC$-linear Tannakian categories. Hence, we have a canonical isomorphism
\[
\mathrm{comp}_{\dR,\rmB}\colon \calV_{N,\an}\xrightarrow{\ \ \sim\ \ }\calL_N\otimes_{\bfQ}\calO_{Y_{\an}}
\]
of flat connections on $Y_{\an}$, which sends $t_N$ to $s_N$.

In~\cite{HainZucker}, Hain and Zucker constructed $\calL_N$ and $\calV_N$ using reduced bar complexes.
By this method, they defined the Hodge filtration $F^\bullet\calV_{N,\an}$ and
the weight filtrations $W_\bullet\calL_N$ and $W_\bullet\calV_{N,\an}$.
\begin{theorem}[{\cite[Proposition (4.20), Proposition (6.15)]{HainZucker}}]
	The tuple \[
	((\calL_N,W_\bullet\calL_N),(\calV_{N,\an},\nabla_N,F^\bullet\calV_{N,\an},W_\bullet\calV_{N,\an}),\mathrm{comp}_{\dR,\rmB})
	\]
	forms a graded-polarizable variation of mixed Hodge structures over $Y_{\an}$.
	Moreover, this graded-polarizable variation of mixed Hodge structures is admissible in the sense of~\cite[Definition 14.48]{PP}.
\end{theorem}
\begin{corollary}
	\label{cor1-1}Let $z$ be a base point of $Y$ which may be tangential when $Y$ is an affine smooth curve. Then there exists a natural mixed Hodge structure on the abelian group $\bfZ[\pi_1^{\mathrm{top}}(Y;y,z)]/\bfZ[\pi_1^{\mathrm{top}}(Y;y,z)]I^N$,
	where $I$ is the augmentation ideal of the group ring $\bfZ[\pi_1^{\mathrm{top}}(Y,y)]$.
\end{corollary}
\begin{remark}
	In~\cite{Wojtkowiak}, Wojtkowiak constructed $\calV_N$ using certain cosimplicial scheme that serves an algebraic analogue of the path space. As an application it can be shown that the filtrations
	$F^\bullet$ and $W_\bullet$ on $\calV_{N,\an}$ descend to $\calV_N$.
\end{remark}
\section{A generalization of the Bloch--Winger function}\label{BWfuct}
In this section, we introduce a generalization of the Bloch--Wigner function following ideas of
Hain and Brown.
We use the same notation as before,
but from now on  we assume that $F$ is a subfield of $\bfR$.
The base point $y$ need not be $F$-rational.
We also fix an $F$-rational base point $b$ of $Y$, which may be tangential.

Define $(\calV_\infty,\nabla_\infty)$ to be $\varprojlim_N(\calV_N,\nabla_N)$.
The canonical extension of $(\calV_N,\nabla_N)$ is denoted by $(\overline{\calV}_N,\overline{\nabla}_N)$ (\cite[2.2]{Del97}),
and the pro-object $\varprojlim_N(\overline{\calV}_N,\overline{\nabla}_N)$ is denoted
by $(\overline{\calV}_\infty,\overline{\nabla}_\infty)$.
We fix a marking $t_\infty=(t_N)_N\in\calV_{\infty,b}$ of $\calV_{\infty}$, \emph{at $b$}.

\subsection{Hain's trivialization}\label{trivializations}
In this subsection, we recall Hain's trivialization of the pro-flat
connection $(\calV_\infty,\nabla_\infty)$.
For a positive integer $N$, define the $F$-vector space $V_N^{\dR}$ by
\[
V_N^{\dR}=\bigoplus_{i=0}^{N-1}\rmH_1^{\dR}(Y/F)^{\otimes i},
\]
where $\rmH_1^{\dR}(Y/F)$ is the $F$-dual of the algebraic de Rham cohomology group $\rmH^1_{\dR}(Y/F)$.
We equip $\rmH_1^{\dR}(Y/F)^{\otimes i}\otimes_F\bfC$ with the usual mixed Hodge structure, and define a mixed Hodge structure on $V_N^{\dR}$ as the direct sum of these mixed Hodge structures.
We regard $V_N^{\dR}$ as a quotient algebra of the tensor algebra $T(\rmH_1^{\dR}(Y/F))=\bigoplus_{i\geq 0}\rmH_1^{\dR}(Y/F)^{\otimes i}$.
The \emph{completed tensor algebra} $\widehat{T}(\rmH_1^{\dR}(Y/F))$ is defined as the projective limit
of $V_N^{\dR}$:
\[
\widehat{T}(\rmH_1^{\dR}(Y/F)):=\varprojlim_NV_N^{\dR}=\prod_{i=0}^\infty\rmH_1^{\dR}(Y/F)^{\otimes i}.
\]

Let $\calA^{p,q}_X$ denote the sheaf of smooth $(p,q)$-forms on $X(\bfC)$ and define a complex $\calA^\bullet_X(\log D)$ of sheaves on $X(\bfC)$ by
\[
\calA^\bullet_X(\log D)=\mathrm{tot}\left(\Omega_{X_{\an}}^\bullet(\log D)\otimes_{\calO_{X_\an}}\calA^{0,*}\right). 
\]
Here, $\mathrm{tot}$ denotes the total complex of the given double complex.
The Hodge and the weight filtrations on this complex is defined in the usual way.
That is, $F^\bullet$ is defined by
\[
F^i\calA^\bullet_X(\log D)=\mathrm{tot}\left(\Omega_{X_{\an}}^{\bullet\geq i}(\log D)\otimes_{\calO_{X_\an}}\calA^{0,*}\right)
\]
and $W_\bullet$ is given by
\[
0=W_{-1}\subset W_0\calA^\bullet_X(\log D)=\calA_X^\bullet\subset W_1\calA^\bullet_X(\log D)=\calA^\bullet_X(\log D).
\]
Note that we have a short exact sequence
\begin{equation}
	\label{eq3}
	0\to  W_0\calA^\bullet_X(\log D)\to  \calA^\bullet_X(\log D)\xrightarrow{\mathrm{Res}} \bigoplus_{c\in D}\bfC\frac{dq_c}{q_c}[1]\to 0,
\end{equation}
where $q_c$ is a local parameter at $c$ and $\bfC\frac{dq_c}{q_c}$ is the skyscraper sheaf supported at $c$.
It is well known that $A^\bullet(X\log D):=\Gamma(X,\calA^\bullet_X(\log D))$ forms the complex part of a mixed
Hodge complex computing the mixed Hodge structure on $\rmH^\bullet(Y(\bfC),\bfZ)$
(see~\cite[Theorem 4.2, Proposition-Definition 4.11]{PP}; cf.\ ~\cite[Theorem 8.35]{Voison}).

The completed tensor product $\widehat{T}(\rmH_1^{\dR}(Y/F))\widehat{\otimes}_FA^\bullet(X\log D)$ is defined by
\[
\widehat{T}(\rmH_1^{\dR}(Y/F))\widehat{\otimes}_FA^\bullet(X\log D)=\varprojlim_N\left(V_N^{\dR}{\otimes}_FA^\bullet(X\log D)\right).
\]
Then, for any two elements $v\otimes \omega,w\otimes \eta\in \widehat{T}(\rmH_1^{\dR}(Y/F))\widehat{\otimes}_FA^\bullet(X\log D)$,
we set
\[
(v\otimes \omega)\wedge (w\otimes \eta):=vw\otimes (\omega\wedge\eta).
\]

Hain defined a one-form $\Omega\in F^0W_{-1}\left(\widehat{T}(\rmH_1^{\dR}(Y/F))\widehat{\otimes}_FA^\bullet(X\log D)\right)$ satisfying the equation
\[
d\Omega-\Omega\wedge\Omega=0
\]
in the following inductive way (see~\cite[Subsection 7.3]{Hain15}).
First, take a closed one-form $\Omega_1\in F^0W_{-1}(\rmH_1^{\dR}(Y/F)\otimes_F A^1(X\log D))\subset V_1^{\dR}\otimes_FA^1(X\log D)$ whose cohomology class represents the identity map on $\rmH^1_{\dR}(Y/\bfC)=\rmH^1(X,\Omega^\bullet_X(\log D))$.
Next, suppose that we have an element $\Omega_N\in F^0W_{-1}(V_N^{\dR}\otimes_FA^1(X\log D))$
satisfying $d\Omega_N-\Omega_N\wedge\Omega_N=0$ in $V_N^{\dR}\otimes_FA^2(X\log D)$.
Then, in $V_{N+1}^{\dR}\otimes_FA^2(X\log D)$, this two-form has a value in $\rmH_1^{\dR}(Y)^{\otimes(N+1)}\otimes_FA^2(X\log D)$.
We can take a one-form $\Omega^{(N+1)}\in F^0W_{-1}(\rmH_1^{\dR}(Y)^{\otimes(N+1)}\otimes_FA^1(X\log D))$
such that
\[
-d\Omega^{(N+1)}-d\Omega_N+\Omega_N\wedge\Omega_N=0
\]
in $V_{N+1}^{\dR}\otimes_FA^2(X\log D)$ as $Y$ is an affine curve and by~\cite[3.2.8]{Hain87b}.
We then set $\Omega_{N+1}:=\Omega_N+\Omega^{(N+1)}$.
Finally, $\Omega$ is defined to be the inverse limit of $\Omega_N$, namely,
\[
\Omega:=\varprojlim_N\Omega_N\in \varprojlim_N F^0W_{-1}(V_N^{\dR}\otimes_FA^1(X\log D))=F^0W_{-1}\left(\widehat{T}(\rmH_1^{\dR}(Y/F))\widehat{\otimes}_FA^1(X\log D)\right).
\]
Let $\widehat{L}(\rmH_1^{\dR}(Y/F))$ denote the set of lie-like elements of $\widehat{T}(\rmH_1^{\dR}(Y/F))$ with respect to the coproduct
\[
\delta\colon \widehat{T}(\rmH_1^{\dR}(Y/F))\to\widehat{T}(\rmH_1^{\dR}(Y/F))^{\widehat{\otimes}2};\quad v\mapsto v\widehat{\otimes }v\quad \mapfor{v}{\rmH_1^{\dR}(Y/F)}.
\] 
Note that $\widehat{L}(\rmH_1^{\dR}(Y/F))$ coincides with the closure in $\widehat{T}(\rmH_1^{\dR}(Y/F))$ of the free Lie algebra generated
by a basis of $\rmH_1^{\dR}(Y/F)$.
By construction, the one-form $\Omega$ is contained in $\widehat{L}(\rmH_1^{\dR}(Y/F))\widehat{\otimes}_FA^1(X\log D)$.
\begin{example}
	\label{ex1}Let us take $Y\subset X$ to be $\bfP^1_F\setminus\{0,1,\infty\}\subset \bfP^1_F$.
	Take a tangent vector $b=\frac{d}{dt}$ at $0$.
	The symbols $e_0,e_1$ denote the basis of $\rmH_1^{\dR}(Y/\bfQ)$ dual to the basis
	$[dt/t], [dt/(t-1)]$ of $\rmH^1_{\dR}(Y/\bfQ)$.
	Then, the connection $\Omega$ can be written as follows:
	\[
	\Omega=e_0\otimes\frac{dt}{t}+e_1\otimes\frac{dt}{t-1}.
	\]
\end{example}
The connection $\nabla_\Omega$ on the pro-sheaf $\widehat{T}(\rmH_1^{\dR}(Y/F))\widehat{\otimes}_F\calA_X^0$ is defined by
\[
\nabla_\Omega=d-\Omega,
\]
and we write the
$(1,0)$-part and the $(0,1)$-part of $\nabla_\Omega$ as $\nabla^{(1,0)}_{\Omega}$ and $\nabla^{(1,0)}_{\Omega}$, respectively.
Define the pro-connection $(\widehat{\calV}_{\Omega,\mathrm{hol}},\nabla_{\Omega,\mathrm{hol}})$
by
\[
\widehat{\calV}_{\Omega,\mathrm{hol}}:=\left(\widehat{T}(\rmH_1^{\dR}(Y/F))\widehat{\otimes}_F\calA_X^0\right)^{\nabla^{(0,1)}=0},\quad \nabla_{\Omega,\mathrm{hol}}=\nabla^{(1,0)}_\Omega.
\]
\begin{theorem}{{\cite[Lemma 7.5, Theorem 7,15]{Hain15}}}
	\label{thm2-1}The flat connection $(\widehat{\calV}_{\Omega,\mathrm{hol}},\nabla_{\Omega,\mathrm{hol}})$ is a pro-holomorphic flat connection on $X_{\an}$. Moreover, this is isomorphic to the analytification of
	$(\overline{\calV}_\infty,\overline{\nabla}_\infty)$.
	This isomorphism is compatible with Hodge and weight filtrations.
\end{theorem}
\begin{example}
	\label{ex2}We use the same notation as in Example~\ref{ex1}.
	Then, we have $\nabla_{\Omega}=\nabla_{\Omega}^{(1,0)}$.
	Therefore,  the analytification of $(\overline{\calV}_\infty,\overline{\nabla}_\infty)$ is given by
	\[
	\overline{\calV}_{\infty,\an}=\bfQ\la\!\la e_0,e_1\ra\!\ra\widehat{\otimes}_\bfQ\calO_{X_\an},\quad \overline{\nabla}_{\infty,\an}=d-e_0\otimes\frac{dt}{t}-e_1\otimes\frac{dt}{t-1},
	\]
	which is canonically isomorphic to 
	the analytification of the KZ-connection on $\bfP^1_{\bfC}$.% (reference)
\end{example}
According to the universality of $(\overline{\calV}_{\infty,\an},\overline{\nabla}_\infty)$,
there exists a \emph{canonical} isomorphism
\begin{equation}\label{triv1}
(\overline{\calV}_{\infty,\an},\overline{\nabla}_\infty)\isom (\widehat{\calV}_{\Omega,\mathrm{hol}},\nabla_{\Omega\mathrm{hol}})
\end{equation}
of pro-flat connections on $X_{\an}$
which sends the local section $t_\infty$ to $1$.
We also  note that there exists a \emph{canonical} comparison isomorphism
\begin{equation}\label{triv2}
\calL_{\infty}\otimes_{\bfQ}\calO_{Y_\an}\xrightarrow{\ \sim\ }\calV_{\infty,\an}
\end{equation}
of flat connections, which induces an isomorphism
\[
\varprojlim_N\bfC[\pi_1^{\top}(Y,b)]/I_b^{N+1}\xrightarrow{\ \sim\ }\calV_{\infty,b}
\]
sending $1$ to $t_\infty$, where $I_b$ is the augmentation ideal of $\bfZ[\pi_1^{\top}(Y,b)]$.
The composition of (\ref{triv1}) and (\ref{triv2}) then induces a canonical isomorphism
\[
t_{\Omega,y}^b\colon \varprojlim_{N}\bfC[\pi_1(Y;b,y)]/\bfC[\pi_1(Y;b,y)]I_b^{N+1}\isom \widehat{T}(\rmH_{\dR}^1(Y/\bfC)).
\]
The isomorphism \emph{does not depend on the choice of $t_\infty$} because of the normalization of (\ref{triv2}).
We call this isomorphism the \emph{trivialization associated with $\Omega$}.

\subsection{The definition}\label{BWfct}
Let $c_{\Omega,\rmB,y}^b$ denote the complex conjugation on $\widehat{T}(\rmH_1^{\dR}(Y/\bfC))$ with respect to
the real structure
\[
t_{\Omega,y}^b\colon\varprojlim_N\bfR[\pi_1^{\top}(Y;b,y)]/\bfR[\pi_1^{\top}(Y;b,y)]I_b^N\intomap \widehat{T}(\rmH_1^{\dR}(Y/\bfC)).
\]
\begin{definition}\label{defofBWfct}We define an element $I^\Omega(b,y)$ of $\widehat{T}(\rmH_1^{\dR}(Y/\bfC))^\times$
	by the equation
	\[
	I^\Omega(b,y)=c_{\Omega,\rmB,y}^b(1).
	\]
\end{definition}
We regard $y\mapsto I^\Omega(b,y)$ as a $\widehat{T}(\rmH_1^{\dR}(Y/\bfC))$-valued function on $Y(\bfC)$.
\begin{definition}[Generalized Bloch--Wigner functions]\label{defofGBWfct}
	\label{dfn2-1}Let $\calH$ denote a Hall basis of the free Lie algebra generated by a basis of $\rmH_1^{\dR}(Y/\bfC)$ (\cite[Part 1, Section 5]{SerreLie}).
	We then define the complex number $D^\Omega_{h}(b,y)$ for $h\in \calH$ by
	\[
	\log(I^\Omega(b,y))=\sum_{h\in\calH}D^\Omega_{h}(b,y)h.
	\]
\end{definition}
We call the function $D^\Omega_{h}(b,y)$, regarded as a function of $y$, a generalized Bloch--Wigner function.
By definition, it is single-valued.
Proposition~\ref{prop2-2} below implies that the function $y\mapsto D^\Omega_{h}(b,y)$ is a smooth function in $y$.

\subsection{Iterated integrals and the generalized Bloch--Wigner function}
For complex-valued smooth one-forms $\omega_1,\dots,\omega_i$ and for a smooth path $\gamma$ 
on $Y(\bfC)$, we define the iterated integral
$\int_{\gamma}\omega_1\cdots\omega_i$ by
\[
\int_{\gamma}\omega_1\cdots\omega_i=\int_{0\leq t_i\leq t_{i-1}\leq \cdots\leq t_1\leq 1}\gamma^*(\omega_1)(t_1)\cdots\gamma^*(\omega_i)(t_i).
\]
Note that our choice of integration order is the same as~\cite{DRS15}, but is opposite of Hain's convention (e.g.\ ~\cite{Hain87}),
because our path composition law is also the reverse of Hain's law.

Let us fix a one-form \[
\Omega\in \widehat{T}(\rmH_1^{\dR}(Y/F))\widehat{\otimes}_FA^1(X\log D))
\]
satisfying the same properties as in Subsection~\ref{trivializations}.
Then, for a smooth path $\gamma$ from $y$ to $b$, the parallel transport along $\gamma$ and (\ref{triv1}) defines an isomorphism
\[
T(\gamma)\colon \widehat{T}(\rmH_1^{\dR}(Y/\bfC))\isom\widehat{T}(\rmH_1^{\dR}(Y/\bfC)).
\]
Suppose that $b$ lies over an $F$-rational point of $Y$. Then, this isomorphism is given by the \emph{left multiplication} by the element
\begin{equation}
	\label{eq9}
	\sum_{n=0}^\infty\int_{\gamma}\Omega^n
\end{equation}
of $\widehat{T}(\rmH_1(Y/\bfC))$.
Here, the iterated integral $\int_{\gamma}\Omega^n$ is defined by
\[
\int_{\gamma}\Omega^n=\sum_{w_1,\dots,w_n\in B}w_1\cdots w_n\int_{\gamma}\alpha_{\omega_1}\cdots\alpha_{w_n},
\]
where $\Omega=\sum_{w\in B}w\alpha_w$.
Indeed, on the universal covering $\widetilde{Y(\bfC)}$ of $Y(\bfC)$, and for a lift $\tilde y\in \widetilde{Y(\bfC)}$ of $y$, each section
\[
\left(\sum_{n=0}^\infty\int_{\tilde y}^{\tilde b}\Omega^n|_{\widetilde{Y(\bfC)}}\right)v,\quad v\in \widehat{T}(\rmH_1^{\dR}(Y/\bfC))
\]
defines a flat section with respect to $\nabla_{\Omega}$.

When $b$ is tangential, in order to describe $T(\gamma)$,  we follow Deligne's definition of the regularized iterated integral (\cite{Del89}; cf.\ ~\cite[Section 4.2]{Brown17}).
Let $b$ be a non-zero element of $ T_cX$, with $c\in D$, and let
\[
\varphi\colon \bbD\to X_{\an}
\]
be a local isomorphism from an open disc $\bbD \subset \bfC$ to $X_{\an}$, centered in $c$, satisfying $d\varphi (0)=b$.
For any path $\gamma$ from $y$ to $b$, and for any sufficiently small positive real number $\epsilon$, let $\gamma_\epsilon$ denote the subpath of $\gamma$ starting from $\varphi(\epsilon)$.
The residue map associated to $b$
\[
\mathrm{Res}_b\colon \widehat{T}(\rmH_1^{\dR}(Y/F))\widehat{\otimes}_FA^1(X\log D)\to \widehat{T}(\rmH_1^{\dR}(Y/F))\widehat{\otimes}_F\bfC
\]
is defined to be the coefficient of $dq_c/q_c$ in the homomorphism induced by the residue map in (\ref{eq3}),
where $q_c$ is a local parameter at $c$ such that $\partial/\partial q_c=b$.
Then one can show that the limit
\begin{equation}
	\label{eq5}
	\lim_{\epsilon\to 0}\exp(-\log(\epsilon)\mathrm{Res}_b(\Omega))\left(\sum_{n=0}^\infty\int_{\gamma_\epsilon}\Omega^n\right)
\end{equation}
converges, by analyzing the local behavior of these integrals (see the proof of~\cite[Proposition 15.45]{Del89}).
We shall use the same notation for the regularized iterated integral as for the usual one,
namely,
\[
\sum_{n=0}^\infty\int_{\gamma}\Omega^n:=\lim_{\epsilon\to 0}\exp(-\log(\epsilon)\mathrm{Res}_c(\Omega))\left(\sum_{n=0}^\infty\int_{\gamma_\epsilon}\Omega^n\right),
\]
for any smooth path $\gamma$ from $y$ to $b$.
For a smooth path $\gamma$ \emph{from $b$ to $y$}, we define the regularized iterated integral by
\[
\sum_{n=0}^\infty\int_{\gamma}\Omega^n:=\left(\sum_{n=0}^\infty\int_{\gamma^{-1}}\Omega^n\right)^{-1}
\]
(this definition will be justified by lemma~\ref{lem2-3} below).
We refer to these quantities and their coefficients as \emph{regularized iterated integrals along $\gamma$}.
Since the usual iterated integrals satisfy the shuffle relation, the regularized iterated integrals
also satisfy the shuffle relation (cf.\ ~\cite[Proposition 3.2]{Brown17}).
In other words, this is contained in the set $\exp(\widehat{L}(\rmH_1^{\dR}(Y/\bfC)))$ of group like elements of $\widehat{T}(\rmH_1^{\dR}(Y/\bfC))$.

The regularized iterated integral can be computed as follows
(see~\cite[Section 4.1]{Brown17}).
For one-forms $\alpha_1,\dots,\alpha_n\in A^1(X\log D)$ and a smooth path $\gamma$ from $y$ to $b$, define
\begin{equation}
	\label{eq10}
	\int_\gamma\alpha_{1}\cdots\alpha_{n}:=\lim_{\epsilon \to 0}\sum_{i=0}^n\int_{\epsilon}^1\mathrm{Res}_c(\alpha_1)\cdots\mathrm{Res}_c(\alpha_i)\int_{\gamma_\epsilon}\alpha_{i+1}\cdots\alpha_{n}.
\end{equation}
Then the regularized iterated integral $\sum_{n=0}^\infty\int_\gamma\Omega^n$ can be
computed in the same way as in the case of an ordinary base point.
\begin{example}
	\label{ex3} We use the same notation as in Example~\ref{ex1}.
	Then the residue $\mathrm{Res}_0(\Omega)$ is equal to $e_0$.
	Let $y$ be a real number contained in the open interval $(0,1)$.
	Then the regularized iterated integrals along the straight line from $y$ to $d/dt$ can be computed as follows:
	\begin{equation*}
		\label{eq6}
		\begin{split}
			\sum_{n=0}^\infty\int_{y}^{d/dt}\Omega^n&=\lim_{\epsilon\to 0}\exp(-\log(\epsilon)e_0)\left(\sum_{n=0}^\infty\int_{y}^{\epsilon}\Omega^n\right)\\
			&=\lim_{\epsilon\to 0}\big[\exp(-\log(\epsilon)e_0)\big(1+e_0(\log(\epsilon)-\log(y))+e_1(\log(1-\epsilon)-\log(1-y))\\
			&\hspace{1cm}+e_0e_1(\log(1-y)(\log(y)-\log(\epsilon))+\mathrm{Li}_2(y)-\mathrm{Li}_2(\epsilon))\\
			&\hspace{1cm}+e_1e_0(\mathrm{Li}_2(\epsilon)-\mathrm{Li}_2(y)-\log(1-\epsilon)(\log(y)-\log(\epsilon))+\cdots)\big]\\
			&=\lim_{\epsilon\to 0}(1-\log(y)e_0-(\log(1-y)-\log(1-\epsilon))e_1\\
			&\hspace{1cm}+(\log(1-y)\log(y)+\mathrm{Li}_2(y)-\mathrm{Li}_2(\epsilon)-\log(\epsilon)\log(1-\epsilon))e_0e_1\\
			&\hspace{1cm}(\mathrm{Li}_2(\epsilon)-\mathrm{Li}_2(y)-\log(1-\epsilon)(\log(y)-\log(\epsilon))e_1e_0+\cdots)\\
			&=1-\log(y)e_0-\log(1-y)e_1+(\log(1-y)\log(y)+\mathrm{Li}_2(y))e_0e_1-\mathrm{Li}_2(y)e_1e_0+\cdots.
		\end{split}
	\end{equation*}
	Here $\mathrm{Li}_2(y)$ is the dilogarithm function defined by
	\[
	\mathrm{Li}_2(y)=\sum_{n=1}^\infty\frac{y^n}{n^2},\qquad |y|<1.
	\]
\end{example}
\begin{lemma}[{\cite[Proposition 15.45]{Del89}; cf.\ ~\cite[Section 8.2]{Luo}}]
	\label{lem2-3}Let $\gamma$ be a smooth path from $y$ to the tangent vector $b$.
	Then, the isomorphism\[
	T(\gamma)\colon \widehat{T}(\rmH_1^{\dR}(Y/\bfC))\isom \widehat{T}(\rmH_1^{\dR}(Y/\bfC))
	\]
	 induced by parallel transport along $\gamma$ is given by the \emph{left multiplication} by
	the regularized iterated integral
	\[
	\sum_{n=0}^\infty\int_{\gamma}\Omega^n.
	\]
\end{lemma}
\begin{remark}
	The residue $\mathrm{Res}_b(\Omega)$ coincides with $-\mathrm{Res}(\nabla_{\Omega,\mathrm{hol}})$
	in the sense of~\cite[Section 15]{Del89}.
\end{remark}
We write $c_{\rmB}$ as the automorphism $c_{\Omega,\rmB,b}^b$
of $\widehat{T}(\rmH_1^{\dR}(Y/\bfC))$ defined in the previous subsection.
\begin{proposition}
	\label{prop2-2}Let $\Omega$ be as above.
	Then, for each $y\in Y(\bfC)$, the following identity holds:
	\[
	I^\Omega(b,y)=\left(\sum_{n=0}^\infty\int_{\gamma}\Omega^n\right)c_{\rmB}\left(\sum_{n=0}^\infty\int_{\gamma}\Omega^n\right)^{-1}.
	\]
	Here $\gamma$ is any smooth path \emph{from $b$ to $y$}.
\end{proposition}
\begin{proof}For simplicity, we write $c_{\rmB,z}$ as the isomorphism
	\[
	\omega_z^{\dR}\otimes\bfC\isom \omega_z^{\dR}\otimes\bfC
	\]
	of fiber functors induced by the Betti--de Rham comparison isomorphisms and the complex conjugations with respect to the Betti $\bfR$-structure.
	We also write $T(\gamma)$ as the automorphism $\omega_{\dR,b}\otimes\bfC\isom\omega_{\dR,y}\otimes \bfC$ induced by parallel transport along $\gamma\in \pi_1^{\top}(Y;b,y)$ by abuse of notation.
	If we regard $\gamma$ as an isomorphism of fiber functors, the following diagram commutes:
	\[
	\xymatrix{
	\omega_{\rmB,b}\otimes\bfC\ar[rrr]^{1\otimes \iota}\ar[d]_\gamma& & &\omega_{\rmB,b}\otimes\bfC\ar[d]^\gamma\\
		\omega_{\rmB,y}\otimes\bfC\ar[rrr]^{1\otimes \iota}& & &\omega_{\rmB,y}\otimes\bfC.
	}
	\]
	Here, $\iota$ denotes the usual complex conjugation on $\bfC$.
	Translating this commutative diagram to the de Rham side yields the identity
	\[
	c_{\rmB,y}=T(\gamma)\circ c_{\rmB,b}\circ T(\gamma^{-1})
	\]
	of natural transformations.
	The proposition now follows directly from Lemma~\ref{lem2-3}.
\end{proof}

\subsection{Complex conjugations}\label{cpxconj}
Let $A^\bullet(Y)$ denote the complex of complex-valued smooth differential forms, namely, $A^\bullet(Y):=\Gamma(Y(\bfC),\calA_Y^\bullet)$.
Let $A^\bullet(Y)_{\bfR}$
denote its natural $\bfR$-structure, defined by the subsheaf of real-valued differential forms.
Let $c$ denote the complex conjugation on $A^\bullet(Y)$
with respect to the real structure $A^\bullet(Y)_{\bfR}$.
Note that $c$ preserves the subcomplex $A^\bullet(X)\subset A^\bullet(Y)$, and coincides with the complex conjugation
with respect to the real-valued differential forms $A^\bullet(X)_{\bfR}$ on $X(\bfC)$.
By abuse of notation, we continue to write $c$ for its restriction to $A^\bullet(X)$.
By definition, $c$ interchanges $(p,q)$ forms and $(q,p)$-forms.

Let
\[
F_\infty\colon Y(\bfC)\to Y(\bfC)
\]
be the map induced by the complex conjugation on $\bfC$.
The pull-back $F_\infty^*$ by $F_\infty$ defines an automorphism of $A^\bullet(Y)$.
It also preserves $A^\bullet(X)$ and interchanges $(q,p)$-forms and $(p,q)$-forms.
Therefore, the composition $c\circ F_\infty^*$ of $A^\bullet(X)$ preserves
the Hodge filtration.
Note that for any affine subscheme $U$ of $X$ defined over $F$, and for any $t\in \Gamma(U,\calO_X)\otimes_F\bfC$, we have
\begin{equation}\label{eq4}
	c\circ F^*_\infty(t)=j(t),
\end{equation}
where $j$ is the complex conjugation on $\Gamma(U,\calO_X)\otimes_F\bfC$ defined by the
real structure $\Gamma(U,\calO_X)\otimes_F\bfR$ (Note that $\Gamma(U,\calO_X)$ is not necessarily contained in $A^0(U)_{\bfR}$).
This identity can be verified by embedding $X$ into a projective space $\bfP^M_F$.
Neither $c$ nor $F_\infty^*$ preserve the subcomplex
$A^\bullet(X \log D)\subset A^\bullet(Y)$.
However, their composition does.
\begin{lemma}
	\label{lem2-2}The composition $c\circ F_\infty^*$ preserves
	$A^\bullet(X \log D)$. Moreover, it preserves both the Hodge and the weight filtrations.
\end{lemma}
\begin{proof}
	According to (\ref{eq4}), we have the equation
	\[
	c\circ	F_\infty^*\left(\frac{dq}{q}\right)=\frac{dj(q)}{j(q)}
	\]
	for any local parameter $q$ at a point of $D$.
	Since $j(q)$ is also a local parameter of a point of $D$, the automorphism
	$c\circ F_\infty^*$ of $A^\bullet(Y)$ preserves $A^\bullet(X\log D)$.
	It follows similarly that $c\circ F_\infty^*$ preserves the Hodge and the
	weight filtration.
\end{proof}
\begin{proposition}
	\label{prop2-1}There exists a one-form \[
	\Omega\in F^0W_{-1}(\widehat{T}(\rmH_1^{\dR}(Y/F))\widehat{\otimes}_FA^1(X\log D))
	\]
	such that
	\[
	d\Omega-\Omega\wedge\Omega=0
	\]
	and
	\[
(\mathrm{id}_{\widehat{T}(\rmH_1^{\dR}(Y/F))}\widehat{\otimes}c)\circ F_\infty^*(\Omega)=\Omega.
	\]
\end{proposition}
\begin{proof}For simplicity, we write $c$ for $\mathrm{id}_{\widehat{T}(\rmH_1^{\dR}(Y/F))}\widehat{\otimes}c$.
	We construct a system of one-forms $\Omega_{N}$ by closely following Hain’s construction.
	
	First, we take $\Omega_1$ as before, and replace it with $\frac{1+c\circ F_\infty^*}{2}\Omega_1$. By Lemma~\ref{lem2-2}, this one-form also lies
	in $F^0W_{-1}(V_1^{\dR}\otimes_FA^1(X\log D))$.
	
	Suppose that we have constructed a one-form \[
	\Omega_N\in F^0W_{-1}(V_N^\dR\otimes_FA^1(X\log D))
	\]
	satisfying
	\[
	\Omega_N-\Omega_N\wedge\Omega_N=0\qquad c\circ F_\infty^*(\Omega_N)=\Omega_N.
	\]
	Here, the first identity holds in $V_N^\dR\otimes_FA^2(X\log D)$.
	Then, we may choose \[
	\Omega^{(N+1)}\in F^0W_{-1}(\rmH_1^{\dR}(Y/F)^{\otimes (N+1)})
	\]
	satisfying
	\[
	-d\Omega^{(N+1)}-d\Omega_N+\Omega_N\wedge\Omega_N=0
	\]
	in $V_{N+1}^\dR\otimes_FA^2(X\log D)$.
	Applying $c\circ F_\infty^*$ to this equation, we obtain
	\[
	-d(c\circ F_\infty^*(\Omega^{(N+1)}))-d\Omega_N+\Omega_N\wedge\Omega_N=0.
	\]
	Therefore, replacing $\Omega^{(N+1)}$ with \[
	\frac{1+c\circ F_\infty^*}{2}\Omega^{(N+1)},
	\]
	which also satisfies the filtration condition by Lemma~\ref{lem2-2},
	we obtain the desired one-form $\Omega_{N+1}$.
\end{proof}
By using the trivialization
\[
t_{\Omega,b}^b\colon \varprojlim_N\bfC[\pi_1^{\top}(Y,b)]/I_{b,\bfC}^N\isom \widehat{T}(\rmH_1^{\dR}(Y/\bfC)),
\]
we define the automorphism $\varphi_{\infty}$ of $\widehat{T}(\rmH_1^{\dR}(Y/\bfC))$ induced by
$F_\infty$. Here, $I_{b,\bfC}$ is the augmentation ideal of $\bfC[\pi_1^{\top}(Y,b)]$.
Let $c_{\dR}$ denote the complex conjugation on $\widehat{T}(\rmH_1^{\dR}(Y/\bfC))$
with respect to the real structure $\widehat{T}(\rmH_1^{\dR}(Y/\bfR))$.
At the end of this subsection, we prove the relation between complex conjugations, which will be used later.
\begin{proposition}
	\label{prop2-5}If $\Omega$ satisfies the conditions in Proposition~\ref{prop2-1},
	then the following identity on $\widehat{T}(\rmH_1^{\dR}(Y/\bfC))$ holds$:$
	\[
	c_{\rmB}\varphi_\infty=c_{\dR}.
	\]
\end{proposition}
\begin{proof}
	We prove only the case where the base point is tangential, since the case of an ordinary base point can be proved in a similar (and simpler) way.
	
	Suppose that $b$ is defined by an $F$-rational tangent vector at $d\in D(F)$. By the equation (\ref{eq4}),
	the residue $\mathrm{Res}_d\Omega$ also satisfies the identity
	\[
	F_\infty^*\mathrm{Res}_d\Omega=\overline{\mathrm{Res}_d\Omega},
	\]
	where $F_\infty$ acts as the complex conjugation on the space $T_d X\otimes_F\bfC$.
	Therefore, for each smooth path $\delta$ from $b$ to $b$, we have 
	\[
	\sum_{n=0}^\infty\int_{F_\infty\circ \delta}\Omega^n=\sum_{n=0}^\infty\int_\delta F_\infty^*\Omega^n=\sum_{n=0}^\infty\int_{\delta}c(\Omega)^n=c_{\dR}\left(\sum_{n=0}^\infty\int_{\delta}\Omega^n\right).
	\]
	This implies that the natural homomorphism
	\[
	T\colon \varprojlim_N\bfR[\pi_1^{\mathrm{top}}(Y,b)]/I_b^N\to \widehat{T}(\rmH_1^{\dR}(Y/F))\widehat{\otimes}\bfC;\quad \delta\mapsto \sum_{n=0}^\infty\int_{\delta}\Omega^n
	\]
	satisfies the identity
	\[
	T(F_\infty\circ\delta)=c_{\dR}(T(\delta)).
	\]
	Since the comparison isomorphism is a $\bfC$-linear extension of this map,
 the conclusion of the lemma follows.
\end{proof}

\subsection{Explicit computations at low length}\label{expl}Let $w,w'$ be elements of $ \rmH_1^{\dR}(Y/F)$.
In this subsection, we compute $D^\Omega_w(b,y)$ and $D^{\Omega}_{[w,w']}(b,y)$ explicitly.
The one-form $\Omega$ is always assumed to satisfy the conditions in Proposition~\ref{prop2-1}.

Fix a basis $\calB$ of $\rmH^1_{\dR}(Y/\bfR)$ as follows.
Let $\calB_h$ be a basis of $\rmH^0(X,\Omega^1_{X/F})\otimes_F \bfR$
and let $\calB_{ah}$ denote the set $\{\overline{\omega}\ |\ \omega\in \calB_h\}$.
Then the set $\calB_0:=\calB_h\coprod\calB_{ah}$ forms a basis of $\rmH^1_{\dR}(X/\bfR)$.
For any element $\omega\in\calB_0$,
we write $[\omega]$ for the cohomology class in $\rmH^1_{\dR}(X/\bfR)$
represented by $\omega$. 
Define $\calB_e\subset \rmH^1_{\dR}(Y/\bfR)$ to be a set of lifts of a basis of
\[
\Coker\left(\rmH^1_{\dR}(X/\bfR)\to \rmH^1_{\dR}(Y/\bfR)\right),
\]
consisting of eigenvectors of $\varphi_{\infty}$. Let $\calB_e^{\pm}$ denote
the set of $\pm$-eigenvectors in $\calB_e$.
Note that if every point of $D$ is defined over $\bfR$, then $\calB^+$ is empty.
Finally,  let $\calB$ denote the set $\calB_h\coprod\calB_{ah}\coprod\calB_e$.
We write  $\{[\omega]^\vee\ |\ [\omega]\in\calB\}$ for the dual basis of $\rmH_1^{\dR}(Y/\bfR)$ to $\calB$
and fix a Hall basis $\calH$ of the free Lie algebra generated by this dual basis.

Recall that $\Omega$ is defined as an infinite sum of one-forms:
\[
\Omega=\sum_{n=1}^\infty\Omega^{(n)},\quad \Omega^{(n)}\in \rmH_1^{\dR}(Y/F)^{\otimes n}\otimes_{F}A^1(X\log D).
\]
In this subsection, we suppose that $\Omega^{(1)}$ is of the form
\[
\Omega^{(1)}=\sum_{[w]\in \calB}[\omega]^\vee\otimes \omega.
\]
We write $\Omega^{(2)}$ as follows:
\[
\Omega^{(2)}=\sum_{[\omega],[\omega']\in\calB}[\omega]^\vee[\omega']^\vee\otimes \alpha_{\omega,\omega'}.
\]
When $\omega$ and $\omega'$ are of the same type so that $\omega\wedge\omega'=0$, then we \emph{always} take $\alpha_{\omega,\omega'}=0$.
\begin{remark}
	Since $d\Omega=\Omega\wedge\Omega$, we have that
	\[
	d\alpha_{{\omega},\omega'}=\omega\wedge\omega'.
	\]
	Therefore, the iterated integral
	\[
	\int_{\gamma}(\omega\omega'+\alpha_{{\omega},\omega'})
	\]
	is homotopy invariant. Indeed, the one-form $\omega\int_\gamma \omega'+\alpha_{{\omega},\omega'}$ is closed, since we have
	\[
	d\left(\omega\int \omega'\right)=-\omega\wedge \omega'.
	\]
\end{remark}
For any element $x$ of $\widehat{T}(\rmH_1^{\dR}(Y/\bfC))$ and for a non-negative integer $n$,
we define $x^{(n)}$ to be the element in $\rmH_1^{\dR}(Y/\bfC)^{\otimes n}$ such that
\[
x=\sum_{n=0}^\infty x^{(n)}.
\]
For example, $T(\gamma)$ can be written as
\[
T(\gamma)=1+T(\gamma)^{(1)}+T(\gamma)^{(2)}+\cdots.
\]
\begin{definition}For any non-negative integers $i,j$,
	we define the $\bfC$-linear homomorphism $\varphi_{\infty}^{[i,j]}$ to be the composition
	\[
	\rmH_1^{\dR}(Y/\bfC)^{\otimes i}\hookrightarrow 	\widehat{T}(\rmH_1^{\dR}(Y/\bfC))\xrightarrow{\varphi_\infty}\widehat{T}(\rmH_1^{\dR}(Y/\bfC))\ontomap\rmH_1^{\dR}(Y/\bfC)^{\otimes j},
	\]
	where the first map is the natural inclusion into the $i$th component, and the last map is the projection to the $j$th component.
	When $i=j$, we write $\varphi_{\infty}^{(i)}$ for $\varphi_{\infty}^{[i,i]}$.
	We also define the semi-linear homomorphisms $c_{\rmB}^{[i,j]}$ and $c_{\rmB}^{(i)}$ by replacing
	$\varphi_{\infty}$ with $c_{\rmB}$.
\end{definition}
\begin{remark}\label{rem1}Let $i$ and $j$ be non-negative integers.
	\begin{itemize}
		\item[(1)]By Proposition~\ref{prop2-5}, $c_{\rmB}^{[i,j]}$ and $\varphi_{\infty}^{[i,j]}$
		coincide on $\rmH_1^{\dR}(Y/\bfR)^{\otimes i}$, and $c_{\rmB}^{[i,j]}$ is the semi-linear extension of
		the restriction of $\varphi_{\infty}^{[i,j]}$ to the real structure above.
		\item[(2)]It is easily seen that $\varphi_{\infty}^{[i,j]}=0$ if $i>j$ or $i=0$.
		\item[(3)]The automorphism $c_{\rmB}^{(i)}$ on $\rmH_1^{\dR}(Y/\bfC)^{\otimes i}$
		coincides with the complex conjugation with respect to the real structure
		\[
		\rmH_1(Y(\bfC),\bfR)^{\otimes i}\intomap		\rmH_1^{\dR}(Y/\bfC)^{\otimes i}.
		\]
		%In particular, we have $c_{\rmB}^{(i)}=(c_{\rmB}^{(1)})^{\otimes i}$ and $\varphi_{\infty}^{(i)}=(\varphi_{\infty}^{(1)})^{\otimes i}$.
	\end{itemize}
\end{remark}
Let us start our computation with the length one case.
\begin{lemma}
	\label{lem2-4}We have the following identity
	:\[
	c_{\rmB}^{(1)}(T(\gamma)^{(1)})=\sum_{[\omega]\in \calB_0}[\omega]^\vee\int_\gamma\omega-\sum_{[\eta]\in\calB_e^-}[\eta]^\vee\overline{\int_\gamma\eta}+\sum_{[\eta]\in\calB_e^+}[\eta]^\vee\overline{\int_\gamma\eta}.
	\]
\end{lemma}
\begin{proof}
	By the remark above, the following holds:
	\[
	\varphi_{\infty}^{(1)}([\omega])=\begin{cases}
		[\overline{\omega}]\quad &[\omega]\in \calB_0,\\
		\pm[\omega]&[\omega]\in \calB_e^{\pm}.
	\end{cases}
	\]
	Then, by Proposition~\ref{prop2-5}, we obtain
	\begin{equation}
		\begin{split}
			c_{\rmB}^{(1)}(T(\gamma)^{(1)})&=c_{\rmB}^{(1)}\left(\sum_{\omega\in\calB}[\omega]^\vee\int_\gamma\omega\right)=\sum_{\omega\in\calB}\varphi_\infty^{(1)}([\omega]^\vee)\overline{\int_\gamma\omega}\\
			&=\sum_{\omega\in \calB_0}[\overline{\omega}]^\vee\int_\gamma\overline{\omega}-\sum_{[\eta]\in\calB_e^-}[\eta]^\vee\overline{\int_\gamma\eta}+\sum_{[\eta]\in\calB_e^+}[\eta]^\vee\overline{\int_\gamma\eta}.
		\end{split}
	\end{equation}
	This completes the proof.
\end{proof}
Let $\varepsilon$ be an element of $ \{+,-\}$. Then, for any complex number $z$,
we define $\calR_{\varepsilon}(z)$ by
\begin{equation}
	\label{sign}
	\calR_{\varepsilon}(z)=z\varepsilon\overline z.
\end{equation}
For $\eta\in \calB_e^{\pm}$, set $\varepsilon(\eta):=\pm$.
\begin{proposition}[The length-one case]
	\label{prop2-3}The following equations hold for any smooth path $\gamma$ from $b$ to $y:$
	\[
	D^\Omega_{[\omega]}(b,y)=\begin{cases}
		0\qquad&[\omega]\in \calB_0,\\
		\calR_{-\varepsilon(\omega)}\left(\int_\gamma\omega\right)&[\omega]\in\calB_e.
	\end{cases}
	\]
\end{proposition}
\begin{proof}According to Proposition~\ref{prop2-2}, we have
	\[
	D^\Omega_{[\omega]^\vee}(b,y)=\la[\omega],T(\gamma)^{(1)}\ra-\la[\omega],c_{\rmB}^{(1)}T(\gamma)^{(1)}\ra
	\]
	(see also Remark~\ref{rem1} (2)).
	Therefore, the conclusion follows from Lemma~\ref{lem2-4}.
\end{proof}
\begin{example}
	\label{ex4}We use the same notation as in Example~\ref{ex1} and Example~\ref{ex3}.
	Note that we have $\varphi_\infty(e_i)=-e_i$.
	Therefore, by the proposition above, we have
	\[
	D^\Omega_{e_i}(b,y)=2\log|i-y|
	\]
	for $i=0,1$.
	Those functions are indeed a single-valued function on $Y(\bfC)$.
\end{example}
Next, we compute the length-two case.
For two elements $\eta_1$ and $\eta_2$ of $\calB_e$, define the sign $\varepsilon(\eta_1,\eta_2)\in \{+,-\}$ to be the product $\varepsilon(\eta_1)\varepsilon(\eta_2)$.
\begin{lemma}
	\label{lem2-5}Let $\gamma$ be a smooth path from $b$ to $y$.
	Then we have the following identity:
	\begin{multline*}
		c_{\rmB}^{(2)}(T(\gamma)^{(2)})=\sum_{[\omega_1],[\omega_2]\in \calB_0}[\omega_1]^\vee[\omega_2]^\vee\int_\gamma(\omega_1\omega_2+\overline{\alpha_{\overline{\omega_1},\overline{\omega_2}}})+\sum_{[\eta_1],[\eta_2]\in \calB_e}\varepsilon(\eta_1,\eta_2)[\eta_1]^\vee[\eta_2]^\vee\overline{\int_\gamma\eta_1\eta_2}\\
		+\sum_{[\omega]\in\calB_0,[\eta]\in \calB_e}\varepsilon(\eta)\left([\overline{\omega}]^\vee[\eta]^\vee\overline{\int_\gamma(\omega\eta+\alpha_{\omega,\eta})}+[\eta]^\vee[\overline{\omega}]^\vee\overline{\int_\gamma(\eta\omega+\alpha_{\eta,\omega})}\right)
	\end{multline*}
\end{lemma}
\begin{proof}
	By the lemma~\ref{lem2-3}, we have
	\begin{multline}
		T(\gamma)^{(2)}=\sum_{[\omega_1],[\omega_2]\in \calB_0}[\omega_1]^\vee[\omega_2]^\vee\int_\gamma(\omega_1\omega_2+\alpha_{\omega_1,\omega_2})+\sum_{[\eta_1],[\eta_2]\in \calB_e}[\eta_1]^\vee[\eta_2]^\vee\int_\gamma\eta_1\eta_2\\
		+\sum_{[\omega]\in\calB_0,[\eta]\in \calB_e}\left([\omega]^\vee[\eta]^\vee{\int_\gamma(\omega\eta+\alpha_{\omega,\eta})}+[\eta]^\vee[{\omega}]^\vee{\int_\gamma(\eta\omega+\alpha_{\eta,\omega})}\right).
	\end{multline}
	Then the conclusion follows in the same manner as the proof of Lemma~\ref{lem2-4}.
\end{proof}
\begin{lemma}
	\label{lem2-6}Let $\xi$ and $\eta$ be elements of $\widehat{T}(\rmH_1^{\dR}(Y/\bfC))$, and let
	$[\omega]$ and $[\omega']$ be elements of $\calB$.
	Suppose that both $\xi$ and $\eta$ are congruent to $1$ modulo the augmentation ideal $\prod_{n\geq 1}\rmH_1^{\dR}(Y/\bfC)$ of $\widehat{T}(\rmH_1^{\dR}(Y/\bfC))$.
	Then the following identity holds:
	\[
	\la[\omega]\otimes[\omega'],\xi\eta^{-1}\ra=\la[\omega]\otimes[\omega'],\xi-\eta\ra-\la[\omega],\xi-\eta\ra\la[\omega'],\eta\ra.
	\]
\end{lemma}
\begin{proof}It is easily checked that the identity
	\[
	\la[\omega]\otimes[\omega'],\xi\eta^{-1}\ra=\la[\omega]\otimes[\omega'],\xi\ra+\la[\omega],\xi\ra\la[\omega'],\eta^{-1}\ra+\la[\omega]\otimes[\omega'],\eta^{-1}\ra
	\]
	holds under our assumption on $\xi$ and $\eta$.
	Moreover, we have the identities
	\[
	\la[\omega'],\eta^{-1}\ra=-\la[\omega'],\eta\ra
	\]
	and
	\[
	\la [\omega]\otimes[\omega'],\eta^{-1}\ra=\la [\omega],\eta\ra\la [\omega'],\eta\ra-\la [\omega]\otimes [\omega'],\eta\ra.
	\]
	Combining these gives the desired result.
\end{proof}
\begin{proposition}
	\label{prop2-4}We use the same notation as above.
	Then we have the following identity$:$
	\begin{multline}
		\la[\omega]\otimes[\omega'],I^\Omega(b,y)\ra=\int_\gamma(\omega\omega'+\alpha_{\omega,\omega'})-\la[\omega]\otimes[\omega]',\varphi_{\infty}^{[1,2]}c_{\dR}(T(\gamma)^{(1)})\ra\\
		-\begin{cases}
			\int_\gamma(\omega\omega'+\overline{\alpha_{\overline{\omega},\overline{\omega}'}})\hspace{2.5cm}&[\omega],[\omega']\in\calB_0,\\
			\varepsilon(\omega,\omega')\overline{\int_\gamma \omega\omega'}+\varepsilon(\omega')D^\Omega_{[\omega]}(b,y)
			\overline{\int_\gamma\omega'}&[\omega],[\omega']\in\calB_e,\\
			\varepsilon(\omega')\overline{\int_\gamma(\overline{\omega}\omega'+\alpha_{\overline{\omega},\omega'})}&[\omega]\in \calB_0,\ [\omega']\in \calB_e,\\
			\varepsilon(\omega)\overline{\int_\gamma(\omega\overline{\omega'}+\alpha_{\omega,\overline{\omega}'})}+D^\Omega_{[\omega]}(b,y)
			{\int_\gamma\omega'}&[\omega]\in \calB_e,\ [\omega']\in \calB_0.
		\end{cases}
	\end{multline}
\end{proposition}
\begin{proof}
	According to Proposition~\ref{prop2-3} and Lemma~\ref{lem2-6}, we have the following equation:
	\begin{multline}
		\label{eq7}
		\la[\omega]\otimes[\omega'],I^\Omega(b,y)\ra=\la[\omega]\otimes[\omega'],T(\gamma)-c_{\rmB}(T(\gamma))\ra\\-\la[\omega],T(\gamma)-c_{\rmB}(T(\gamma))\ra\la[\omega'],c_{\rmB}(T(\gamma))\ra.
	\end{multline}
	Since $\la[\omega],T(\gamma)-c_{\rmB}(T(\gamma))\ra=D^\Omega_{[\omega]}(b,y)$, the second term in (\ref{eq7}) can be computed as follows:
	\begin{equation}
		\label{eq8}
		\la[\omega],T(\gamma)-c_{\rmB}(T(\gamma))\ra\la[\omega'],c_{\rmB}(T(\gamma))\ra=D^\Omega_{[\omega]}(b,y)\times\begin{cases}
			{\int_\gamma\omega'}\qquad &[\omega']\in\calB_0,\\
			\varepsilon(\omega')\overline{\int_\gamma\omega'} &[\omega']\in\calB_e.
		\end{cases}
	\end{equation}
	Moreover, since the identity
	\[
	\la[\omega]\otimes[\omega]',c_{\rmB}(T(\gamma))=\la[\omega]\otimes[\omega]',c_{\rmB}^{(2)}(T(\gamma)^{(2)})\ra+\la[\omega]\otimes[\omega]',\varphi_{\infty}^{[1,2]}(T(\gamma)^{(1)})\ra
	\]
	holds, the first term in~\ref{eq7} can be evaluated using Lemma~\ref{lem2-5} as follows:
	\begin{equation}
		\begin{split}
			\la[\omega]\otimes[\omega]',T(\gamma)-c_{\rmB}(T(\gamma))\ra&=\int_\gamma(\omega\omega'+\alpha_{\omega,\omega'})-\la[\omega]\otimes[\omega]',c_{\rmB}(T(\gamma))\ra\\
			&=\int_\gamma(\omega\omega'+\alpha_{\omega,\omega'})-\la[\omega]\otimes[\omega]',\varphi_{\infty}^{[1,2]}c_{\dR}(T(\gamma)^{(1)})\ra\\
			&-\begin{cases}
				\int_\gamma(\omega\omega'+\overline{\alpha_{\overline{\omega},\overline{\omega}'}})\qquad&[\omega],[\omega']\in\calB_0,\\
				\varepsilon(\omega,\omega')\overline{\int_\gamma \omega\omega'}&[\omega],[\omega']\in\calB_e,\\
				\varepsilon(\omega')\overline{\int_\gamma(\overline{\omega}\omega'+\alpha_{\overline{\omega},\omega'})}&[\omega]\in \calB_0,\ [\omega']\in \calB_e,\\
				\varepsilon(\omega)\overline{\int_\gamma(\omega\overline{\omega'}+\alpha_{\omega,\overline{\omega}})}&[\omega]\in \calB_e,\ [\omega']\in \calB_0.
			\end{cases}
		\end{split}
	\end{equation}
	Combining these computations completes the proof of the proposition.
\end{proof}
\begin{theorem}[The length-two case]
	\label{thm2-2}Let $[\omega],[\omega']$ be elements of $\calB$. Then the following identity holds$:$
	\begin{multline*}
		D_{[[\omega]^\vee,[\omega']^\vee]}^{\Omega}(b,y)=-\la[\omega]\otimes[\omega'],\varphi_{\infty}^{[1,2]}c_{\dR}(T(\gamma)^{(1)})\ra\\
		+\begin{cases}
			\int_{\gamma}(\alpha_{{\omega},\omega'}-\overline{\alpha_{\overline{\omega},\overline{\omega}'}})\hspace{6cm}&[\omega],[\omega']\in\calB_0,\\
			\calR_{-\varepsilon(\omega,\omega')}\left(\int_\gamma \omega\omega'\right)-\frac{1}{2}\calR_{-\varepsilon(\omega)}\left(\int_\gamma\omega\right)\calR_{\varepsilon(\omega')}\left(\int_\gamma\omega'\right)&[\omega],[\omega']\in\calB_e,\\
			\int_\gamma(\omega\calR_{-\varepsilon(\omega')}(\omega')-\epsilon(\omega')\overline{\alpha_{\overline{\omega},\omega'}}))&[\omega]\in \calB_h,\ [\omega']\in \calB_e,\\
			\int_\gamma(\overline{\omega}\calR_{-\varepsilon(\omega')}(\omega')+\alpha_{{\omega},\omega'}))&[\omega]\in \calB_{ah},\ [\omega']\in \calB_e.
			%\\
			%-\overline{\int_\gamma(\omega\overline{\omega'}+\alpha_{\omega,\overline{\omega}})}+2\mathrm{Re}\left(\int_\gamma\omega\right)
			%\overline{\int_\gamma\omega}&[\omega]\in \calB_e,\ [\omega']\in \calB_e.
		\end{cases}
	\end{multline*}
\end{theorem}
\begin{proof}
	By definition, the function $D_{[[\omega]^\vee,[\omega']^\vee]}^{\Omega}(b,y)$ coincides with the coefficient of $[\omega]^\vee[\omega']^\vee$ in the power series
	\[
	(I^\Omega(b,y)-1)-\frac{1}{2}(I^\Omega(b,y)-1)^2.
	\]
	Therefore, we have
	\begin{multline}
		D_{[[\omega]^\vee,[\omega']^\vee]}^{\Omega}(b,y)=\la[\omega]\otimes[\omega'],I^\Omega(b,y)\ra-\frac{1}{2}\la[\omega],I^\Omega(b,y)\ra\la[\omega'],I^\Omega(b,y)\ra\\
		=\la[\omega]\otimes[\omega'],I^\Omega(b,y)\ra-\frac{1}{2}D^{\Omega}_{[\omega]^\vee}(y)D^{\Omega}_{[\omega']^\vee}(y).
	\end{multline}
	The second term of the above expression vanishes unless both $\omega$ and $\omega'$
	are elements of $\calB_e$.
	Thus, the theorem follows directly from Proposition~\ref{prop2-4}.
\end{proof}
\begin{example}
	We follow notations introduced in Examples~\ref{ex1}--\ref{ex4}.
	In this case, it is known that the homomorphism $\varphi_{\infty}^{[1,2]}$ is equal to zero.
	
	For a smooth path $\gamma$ \emph{from $d/dt$ to $y$}, the regularized iterated integral of $\frac{dt}{t}\frac{dt}{t-1}$ along $\gamma$ is given by
	\[
	\int_\gamma\frac{dt}{t}\frac{dt}{t-1}=-\mathrm{Li}_2(y).
	\]
	Therefore, applying Proposition~\ref{prop2-4}, we obtain
	\begin{equation}
		\begin{split}
			D^\Omega_{[e^0, e^1]}(b,y)
			&=-2\sqrt{-1}\mathrm{Im}(\mathrm{Li}_2(y))-2\sqrt{-1}\log|y|\mathrm{arg}(1-y)\\
			&=-2\sqrt{-1}D(y),
		\end{split}
	\end{equation}
	where $D(y)$ is the classical Bloch--Wigner function (\cite[Section 2]{Zagier}).
\end{example}

%%%%%%%%%%%%%%%%%%%%%%%%%%
%%%%%%%%%%%%%%%%%%%%%%%%%%
%%%%%%%%%%%%%%%%%%%%%%%%%%
\section{The matrix coefficient formulas}\label{theformula}
In this section, we compute the matrix coefficients
\[
\la[\omega]\otimes[\omega'],\varphi_\infty^{[1,2]}([\eta]^\vee)\ra
\]
of the infinite Frobenius $\varphi_{\infty}$ by analyzing Theorem~\ref{thm2-2}.
We divide the computation into three cases.
The first case is when both $[\omega]$ and $ [\omega']$ are in $\calB_0$; the second case is
when one of $[\omega],[\omega']$ is in $\calB_h$; and the third case is when both $[\omega]$ and $ [\omega']$ are in $\calB_e $.
In the first and second cases, we sometimes replace $\alpha_{{\omega},\omega'}$
with more suitable choices.

The basis $\calB=\calB_h\coprod\calB_{ah}\coprod\calB_e$ of $\rmH^1_{\dR}(Y/\bfR)$, and the one-form $\Omega=\Omega^{(1)}+\Omega^{(2)}+\cdots$, are taken as in Subsection~\ref{expl}.

\subsection{The first case}
Suppose that both  $[\omega]$ and $[\omega']$ are elements of $\calB_0$ and of the same type,
namely, both holomorphic or both anti-holomorphic.
Recall that in this case, we take $\alpha_{{\omega},\omega'}$ to be zero.
Set $w=[\omega]^\vee$ and $w'=[\omega']^\vee$.
Then, for any smooth path $\gamma$ from $b$ to $y$, we have
\begin{equation}\label{eq14}
	D^{\Omega}_{[w,w']}(b,y)=\la[\omega]\otimes[\omega'],\varphi_{\infty}^{[1,2]}c_{\dR}(T(\gamma)^{(1)})\ra
	=\sum_{\eta\in\calB}\la[\omega]\otimes[\omega'],\varphi_{\infty}^{[1,2]}([\eta]^\vee)\ra\overline{\int_\gamma\eta}
\end{equation}
by Theorem~\ref{thm2-2} and Lemma~\ref{lem2-4}.
\begin{theorem}
	\label{thm3-0}Suppose that both $[\omega]$ and $[\omega']$ are elements of $\calB_0$ and of the same type.
	Then for any $\eta\in \calB$, we have
	\[
	\la[\omega]\otimes[\omega'],\varphi_\infty^{[1,2]}([\eta]^\vee)\ra=0.
	\]
\end{theorem}
\begin{proof}By equation (\ref{eq14}), we have the following identity of smooth one-forms:
	\[
	dD^{\Omega}_{[w,w']}(b,y)=\sum_{\eta\in\calB}\la[\omega]\otimes[\omega'],\varphi_{\infty}^{[1,2]}([\eta]^\vee)\ra\overline{\eta}.
	\]
	Then the conclusion follows from the linear independence of $\eta$ in the cohomology group $\rmH^1_{\dR}(Y/\bfC)$.
\end{proof}
\begin{remark}We have not computed the case where $\omega\in \calB_h$ and $\omega'\in \calB_{ah}$, which is the remaining subcase of the first case.
\end{remark}

\subsection{The second case}\label{secondcase}
Next, we consider the case where $[\omega]\in \calB_0$ and $[\omega']\in \calB_e$.
Here, we compute only the case where $[\omega]\in \calB_h$, since the same result holds for $[\omega]\in \calB_{ah}$ by a similar argument. 
Since the equation $F_\infty^*(\Omega^{(1)})=c(\Omega^{(1)})$ holds and the equation $\varphi_{\infty}([\omega'])=\pm[\omega']$ holds for $\omega'\in \calB_e^{\pm}$, one-forms $\omega'$ and $\pm \overline{\omega'}$ represent the same cohomology class in $\rmH^1_{\dR}(Y(\bfC))$.
Therefore, the one-form $\calR_{-\varepsilon(\omega')}(\omega')$ is an exact form. Let $l$ be $0$ or $1$ according as $\epsilon=-$ or $+$, respectively. We choose an $\bfR(l)$-valued function $\calE_{\omega'}$
on $Y(\bfC)$ with at worst logarithmic singularities along $D$, satisfying
\begin{equation}
	\label{eq29}
	d\calE_{\omega'}=\calR_{-\varepsilon(\omega')}(\omega').
\end{equation}
Here, we use the term \emph{logarithmic singularity} in the following sense:
Let $f$ be a smooth function on $Y(\bfC)$. We say that $f$ has at worst logarithmic singularities along $D$ if, for each $c\in D$, we have
\[
f(q_c)=a\log|q_c|+O(1)
\]
for some $a\in \bfC$, where $q_c$ is any local parameter at $c$.

For $\omega\in\calB_h$ and $\omega'\in \calB_e$, we have the following identity of two-forms:
\[
\overline{\omega}\wedge\omega'=\overline{\omega}\wedge \calR_{-\varepsilon(\omega')}=-d(\calE_{\omega'}\overline{\omega}),
\]
since $\overline{\omega}$ is a closed form.
It follows that the one-form $\alpha_{\overline{\omega},\omega'}+\calE_{\omega'}\overline{\omega}$ is closed and belongs to $A^1(Y)$.
Since $\calB$ is a basis of the first cohomology group of the complex $A^\bullet(Y)$,
there exist complex numbers $a_{\eta},\ \eta\in\calB$ and a smooth function $\xi$ on $Y(\bfC)$, such that
\begin{equation}
	\label{eq11}
	\alpha_{\overline{\omega},\omega'}+\calE_{\omega'}\overline{\omega}=\sum_{\eta\in\calB}a_\eta\eta+d\xi.
\end{equation}
Note that, since the differential form $d\xi$ has at worst logarithmic singularities along $D$,
the function $\xi$ itself has at worst logarithmic singularities along $D$.
Applying $c_{\dR}F_\infty^*$ to the both-hand side of (\ref{eq11}),
we obtain
\[
0=2\sqrt{-1}\sum_{\eta}\mathrm{Im}(a_\eta)\eta+d(\xi-c_{\dR}F_\infty^*\xi),
\]
which implies that each $a_\eta$ must be real.
Therefore, we may replace $\alpha_{\overline{\omega},\omega'}$ by
$\alpha_{\overline{\omega},\omega'}-\sum_{\eta\in \calB_h\cup\calB_e}a_\eta\eta$.
After performing this replacement,
we have
\begin{equation}\label{eq12}
	\alpha_{\overline{\omega},\omega'}=-\calE_{\omega'}\overline{\omega}+\sum_{\eta\in \calB_{ah}}a_\eta\eta+d\xi.
\end{equation}

Let $(\ ,\ )$ denote the pairing on smooth one-forms on $X(\bfC)$ defined by
\[
(\alpha,\beta)=\frac{1}{2\pi\sqrt{-1}}\int_{X(\bfC)}\alpha\wedge\overline{\beta}.
\]
It is well known that this pairing defines a positive-definite Hermitian pairing on $\rmH^1_{\dR}(X(\bfC))$.
\begin{lemma}
	\label{lem2-7}Let us choose $\calB_h$ such that the following identities holds:
	\[
	(\eta,\eta')=0
	\]
	for all $\eta,\eta'\in \calB_h$ such that $ \eta\neq \eta'$.
	Then, for each $\omega\in \calB_h$ and $\omega'\in \calB_e$, we can choose $\alpha_{\overline{\omega},\omega'}$ such that
	\[
	\alpha_{\overline{\omega},\omega'}=-\calE_{\omega'}\overline{\omega}-\varepsilon(\omega')\sum_{\eta\in \calB_h}\frac{(\eta,\calE_{\omega'}\omega)}{(\eta,\eta)}\overline{\eta}+d\xi_{\omega,\omega'}
	\]
	for a smooth function $\xi_{\omega,\omega'}$ on $Y(\bfC)$ with at worst logarithmic singularities along $D$.
\end{lemma}
\begin{proof}Let us take $\xi_{\omega,\omega'}$ to be the function $\xi$ appearing in the equation (\ref{eq12}).
	Since $\alpha_{{\omega},\omega'}$ is of type $(1,0)$,
	we have
	\[
	\eta\wedge\calE_{\omega'}\overline{\omega}=\sum_{\eta'\in \calB_h}a_{\eta'}\eta\wedge\overline{\eta'}+\eta\wedge d\xi_{\omega,\omega'}
	\]
	for each $\eta\in \calB_h$.
	For each $c\in D$, let us fix a local parameter $q_c$, and define the open disk 
	\[
	\Delta_{c,\epsilon}:=\{x\in X(\bfC)\ |\ |q_c(x)|<\epsilon\}.
	\]
	Set $Y_\epsilon:=X(\bfC)\setminus \cup_{c\in D}\Delta_{c,\epsilon}$.
	Then, by Stokes' theorem, we have the identity
	\begin{equation}
		\label{vanishingpart}
	\int_{Y_\epsilon}\eta\wedge d\xi=-\sum_{c\in D}\int_{\partial\Delta_{c,\epsilon}}\xi\eta.
	\end{equation}
	We write $q_c=r\exp(2\pi\sqrt{-1}\theta)$.
	On the boundary $\partial\Delta_{c,\epsilon}$ of $\Delta_{c,\epsilon}$, we have
	\[
	\xi\eta=(\alpha\log(\epsilon)+O(1))\epsilon d\theta.
	\]
	Therefore, the integral in (\ref{vanishingpart}) tends to zero as $\epsilon\to 0$.
	
	Hence, we obtain
	\begin{multline}
	(\eta,\calE_{\omega'} \omega)=-\epsilon(\omega')\lim_{\epsilon\to 0}\int_{Y_\epsilon}\eta\wedge\calE_{\omega'}\overline{\omega}=-\epsilon(\omega')\lim_{\epsilon\to 0}\sum_{\eta'}a_{\eta'}\int_{Y_\epsilon}\eta\wedge \overline{\eta'}\\
	=-\epsilon(\omega')\sum_{\eta'}a_{\eta'}(\eta,\eta')=-\epsilon(\omega')a_\eta(\eta,\eta).
	\end{multline}
	Here, we used the identity $\calE_{\omega'}\overline\omega=-\epsilon(\omega')\overline{\calE_{\omega'}\omega}$.
	This proves the lemma.
\end{proof}
\begin{theorem}\label{thm3-1}Let $ \Omega$ and $\calB_h$ be as in Lemma~\ref{lem2-7}.
	Suppose that $b$ lies over an element of $Y(\bfC)$.
	Then, for any $\omega\in\calB_h$ and $[\omega']\in \calB_e$, the following identities hold$:$
	\[
	\la[\omega]\otimes[\omega'],\varphi^{[1,2]}_\infty([\eta]^\vee)\ra=\begin{cases}
		\frac{(\calE_{\omega'}\omega,\overline{\eta})}{(\eta,\eta)}\qquad&\eta\in \calB_{ah},\ \overline{\eta}\neq \omega,\\
		\frac{((\calE_{\omega'}-\calE_{\omega'}(b))\omega,\omega)}{(\omega,\omega)}&\overline{\eta}=\omega,\\
		0&\text{other wise}.
	\end{cases}
	\]
\end{theorem}
\begin{proof}
	Set $w=[\omega]^\vee,\ w'=[\omega']^\vee$, and $\varepsilon':=\varepsilon(\omega')$. Then according to Theorem~\ref{thm2-2} and Lemma~\ref{lem2-4}, we have the following equation:
	\[
	D_{[w,w']}^{\Omega}(y)=\int_\gamma(\omega\calR_{-\varepsilon'}(\omega')-\varepsilon'\overline{\alpha_{\overline{\omega},\omega'}})
	-\sum_{\eta\in\calB}	\la[\omega]\otimes[\omega'],\varphi^{[1,2]}_\infty([\eta]^\vee)\ra\overline{\int_\gamma\eta}
	\]
	for any smooth path $\gamma$ from $b$ to $y$.
	Since the identity
	\[
	\int_b^y\omega\calR_{-\varepsilon(\omega')}(\omega')=\int_b^y\omega(\calE_{\omega'}-\calE_{\omega'}(b))
	\]
	holds, we obtain the following equation of one-forms:
	\[
	dD_{[w,w']}^{\Omega}(y)=\omega(\calE_{\omega'}-\calE_{\omega'}(b))-\varepsilon'\overline{\alpha_{\overline{\omega},\omega'}}
	-\sum_{\eta\in\calB}	\la[\omega]\otimes[\omega'],\varphi^{[1,2]}_\infty([\eta]^\vee)\ra\overline\eta.
	\]
	By Lemma~\ref{lem2-7}, we have
	\[
	\varepsilon'\overline{\alpha_{\overline{\omega},\omega'}}=\calE_{\omega'}\omega-\sum_{\eta\in \calB_h}\frac{(\calE_{\omega'}\omega,\eta)}{(\eta,\eta)}{\eta}+d\xi
	\]
	for some smooth function $\xi$ on $Y(\bfC)$.
	Therefore, the equation
	\begin{multline}\label{eq13}
		d\xi'=
		\left(\frac{(\calE_{\omega'}\omega,\omega)}{(\omega,\omega)}-\calE_{\omega'}(b)-\la[\omega]\otimes[\omega'],\varphi^{[1,2]}_\infty([\overline{\omega}]^\vee)\ra\right)\omega\\
		+\sum_{\eta\in \calB_h\setminus\{\omega\}}\left(\frac{(\calE_{\omega'}\omega,\eta)}{(\eta,\eta)}
		-\la[\omega]\otimes[\omega'],\varphi^{[1,2]}_\infty([\overline{\eta}]^\vee)\ra\right)\eta
		-\sum_{\eta\in\calB\setminus\calB_{ah}}	\la[\omega]\otimes[\omega'],\varphi^{[1,2]}_\infty([\eta]^\vee)\ra\overline\eta
	\end{multline}
	holds for a smooth function $\xi'$ on $Y(\bfC)$.
	By the linear independence of  the forms $\eta$ in the cohomology group $\rmH^1_{\dR}(Y/\bfC)$,
	we obtain the conclusion of the theorem.
\end{proof}
When using the tangential base point, the formula becomes
much simpler.
Suppose that $b$ is a tangent vector at $c\in D$.
We normalize $\calE_{\omega'}$ by the condition
\begin{equation}
	\label{eq18}
	\lim_{\epsilon\to 0}(\calE_{\omega'}(\gamma(\epsilon))-2\mathrm{Res}_b(\omega')\log(\epsilon))=0
\end{equation}
for any smooth path $\gamma$ from $b$ to $y$. This condition does not depend on the choice of $\gamma$.
\begin{theorem}
	\label{thm3-2}We use the same notation as in Theorem~\ref{thm3-1}.
	Suppose that $b$ is a tangential base point of $Y$.
	Then, the following identities hold$:$
	\[
	\la[\omega]\otimes[\omega'],\varphi^{[1,2]}_\infty([\eta])\ra=\begin{cases}
		\frac{(\calE_{\omega'}\omega,\overline{\eta})}{(\eta,\eta)}\qquad&\eta\in \calB_{ah},\\
		0&\text{otherwise}.
	\end{cases}
	\]
\end{theorem}
\begin{proof}By the equation (\ref{eq10}), we can compute the regularized iterated integral as follows
	\begin{multline}
		\int_{\gamma}\omega\omega'+\overline{\int_{\gamma}\overline{\omega}\omega'}
		=\lim_{\epsilon\to 0}\left(\int_{\gamma_\epsilon}(\omega\calE_{\omega'})+\int_{\gamma_\epsilon}\omega\calR_{-\varepsilon(\omega')}\int_{\epsilon}^1\mathrm{Res}_c(\omega')\right)\\
		=\lim_{\epsilon\to 0}\left(\int_{\gamma_\epsilon}(\omega\calE_{\omega'}-2\mathrm{Res}_z(\omega')\log|q_c|)\right).
	\end{multline}
	Let $q_c$ be a local parameter at $c$ satisfying $\partial/\partial q_c=b$.
	Since we have \[
	\lim_{t\to 0}(\calE_{\omega'}(\gamma(t))-2\mathrm{Res}_b(\omega')\log|q_c(\gamma(t))|))=0
	\]
	by our normalization (\ref{eq18}),
	the following equation holds for a smooth function $\xi'$ on $Y(\bfC)$ (cf.\ Equation (\ref{eq13})):
	\begin{multline}\label{eq15}
		d\xi'=
		\sum_{\eta\in \calB_h}\left(\frac{(\calE_{\omega'}\omega,\eta)}{(\eta,\eta)}
		-\la[\omega]\otimes[\omega'],\varphi^{[1,2]}_\infty([\overline{\eta}]^\vee)\ra\right)\eta
		\\
		-\sum_{\eta\in\calB\setminus\calB_{ah}}	\la[\omega]\otimes[\omega'],\varphi^{[1,2]}_\infty([\eta]^\vee)\ra\overline\eta.
	\end{multline}
	Hence, we obtain the conclusion of the theorem by an argument similar to that of Theorem~\ref{thm3-1}.
\end{proof}

\subsection{The third case}\label{thirdcase}
As in the previous subsection, we take a function $\calE_{\omega}$ for each $\omega\in \calB_e$. We define a one-form $\calF_{\omega,\omega'}$, having at worst logarithmic singularities along $D$, by
\begin{equation}
	\label{eq27}
	\calF_{\omega,\omega'}=\calE_{\omega'}\calR_{\varepsilon(\omega)}(\omega)-\calE_{\omega}\calR_{\varepsilon(\omega')}(\omega')
\end{equation}
(cf.\ ~\cite[9.2.3]{Brown17}).
It is easy to check that $\calF_{\omega,\omega'}$ is closed.
\begin{lemma}
	\label{lem3-3}Let $w=[\omega]^\vee$ and $w'=[\omega']^\vee$. If $b$ lies over an $F$-rational point of $Y$, then the following identity of differential one-forms on $Y(\bfC)$ holds$:$
	\begin{equation}\label{eq17}
		d D^{\Omega}_{[w,w']}=\frac{1}{2}\left(\calF_{\omega,\omega'}-\calE_{\omega'}(b)\calR_{\epsilon}(\omega)+\calE_{\omega}(b)\calR_{\varepsilon'}\left(\omega'\right)\right)+\sum_{\eta\in\calB}\la[\omega]\otimes[\omega'],\varphi^{[1,2]}_\infty([\eta]^\vee)\ra\overline{\eta}.
	\end{equation}
\end{lemma}
\begin{proof}Put $\varepsilon:=\varepsilon(\omega)$ and $\varepsilon':=\varepsilon(\omega')$.
	Then, by $d\int_b^y\alpha\beta=\alpha\int_b^y\beta$, we obtain
	\begin{multline}
		d\calR_{-\varepsilon\varepsilon}\left(\int_\gamma \omega\omega'\right)=\calR_{-\varepsilon\varepsilon'}\left(\omega\int_b^y \omega'\right)\\
		=\frac{1}{2}\calR_{-\varepsilon}(\omega)\calR_{\epsilon'}\left(\int_b^y\omega'\right)+\frac{1}{2}\calR_{\varepsilon}\left(\omega\right)\calR_{-\varepsilon'}\left(\int_b^y\omega'\right).
	\end{multline}
	Therefore, we have the following identities:
	\begin{equation}
		\begin{split}
			d\left(\calR_{-\varepsilon\varepsilon'}\left(\int_b^y \omega\omega'\right)\right.&\left.-\frac{1}{2}\calR_{-\varepsilon}\left(\int_b^y\omega\right)\calR_{\varepsilon'}\left(\int_b^y\omega'\right)\right)\\
			&=\frac{1}{2}\left(\calR_{-\varepsilon}(\omega)\calR_{\varepsilon'}\left(\int_b^y\omega'\right)+\calR_{\varepsilon}(\omega)\calR_{-\varepsilon'}\left(\int_b^y\omega'\right)\right)\\
			&\hspace{1cm}-\frac{1}{2}\left(\calR_{-\varepsilon}\left(\omega\right)\calR_{\varepsilon'}\left(\int_b^y\omega'\right)+\calR_{-\varepsilon}\left(\int_b^y\omega\right)\calR_{\varepsilon'}\left(\omega'\right)\right)\\
			&	=\frac{1}{2}\left(\calR_{\epsilon}(\omega)\calR_{-\varepsilon}\left(\int_b^y\omega\right)-\calR_{-\varepsilon}\left(\int_b^y\omega\right)\calR_{\varepsilon'}\left(\omega'\right)\right)\\
			&	=\frac{1}{2}\left(\calF_{\omega,\omega'}-\calE_{\omega'}(b)\calR_{\epsilon}(\omega)+\calE_{\omega}(b)\calR_{\varepsilon'}\left(\omega'\right)\right).
		\end{split}
	\end{equation}
	The lemma now follows from Theorem~\ref{thm2-2} and the equations above.
\end{proof}
\begin{theorem}
	\label{thm3-3}Suppose that $b$ lies over an $F$-rational point of $Y$. Then, for any $[\omega],[\omega']\in \calB_e$ and $[\eta]\in \calB_0$, we have
	\begin{multline}
		\la[\omega]\otimes[\omega'],\varphi_{\infty}^{[1,2]}([\overline{\eta}]^\vee)\ra=\frac{1}{2}\frac{(\calE_{\omega'}(b)\calR_{\epsilon}(\omega)-\calE_{\omega}(b)\calR_{\varepsilon'}\left(\omega'\right)
			-\calF_{\omega,\omega'},\eta)}{(\eta,\eta)}\\
		-\sum_{\eta'\in\calB_e}\la[\omega]\otimes[\omega'],\varphi_{\infty}^{[1,2]}([\eta']^\vee)\frac{(\overline{\eta'},\eta)}{(\eta,\eta)}.
	\end{multline}
	In particular, if ${\eta}$ is orthogonal to every element of $\calB_e$ with respect to the pairing $(\ ,\ )$, then the following identity holds$:$
	\[
	\la[\omega]\otimes[\omega'],\varphi_{\infty}^{[1,2]}([\overline{\eta}]^\vee)=-\frac{1}{2}\frac{(\calF_{\omega,\omega'},\eta)}{(\eta,\eta)}
	\]
\end{theorem}
\begin{proof}
	Since both sides of identity (\ref{eq17}), when wedged with $\overline{\eta}$, have at worst logarithmic singularities, their integrals over $Y(\bfC)$ converge. 
	By Stokes' theorem  (cf.\ the proof of Lemma~\ref{lem2-7}), we obtain
	\[
	\int_{Y(\bfC)}(d D^{\Omega}_{[w,w']}\wedge \overline{\eta})=0.
	\]
	The theorem then follows.
\end{proof}
If we normalize $\calE_{\omega}$ and $\calE_{\omega'}$ as in the previous
subsection, we obtain the equation for the tangential bae point case
by replacing $\calE_{\omega}(b)$ and $\calE_{\omega'}(b)$ with zero.
The computation is very similar to the proof of Theorem~\ref{thm3-2}, so we omit the proof of
this formula.
\begin{remark}
	The matrix coefficient $\la[\omega]\otimes[\omega'],\varphi_{\infty}^{[1,2]}([\eta]^\vee)\ra$
	for $[\omega],[\omega'],[\eta]\in\calB_e$ has not been computed yet.
	This is the remaining subcase of the third case.
\end{remark}
\section{The regulator formulas}\label{reglatorformula}
Let $I_b$ be the augmentation ideal of $\bfQ[\pi_1^{\top}(Y,b)]$.
In this section, we partially determine the isomorphism class of the extension
\[
0\to \rmH_1(Y(\bfC),\bfR)^{\otimes 2}\to I_b/I_b^3\otimes_{\bfQ}\bfR\to \rmH_1(Y(\bfC),\bfR)\to 0
\] 
of $\bfR$-mixed Hodge structures,
based on the computations given in Section~\ref{theformula}.

\subsection{Mixed Hodge structures with infinite Frobenius}\label{MHSwithfrob}
Let $A$ be a subalgebra of $\bfR$. An $A$-mixed Hodge structure with infinite Frobenius (\cite[(2.4)]{Nekovar})
is a tuple \[
H=(H_{\rmB},\varphi_\infty,F^\bullet H_{\rmB,\bfC},W_\bullet H_{\rmB})
\]
where
\begin{itemize}
	\item $(H_{\rmB},F^\bullet H_{\rmB,\bfC},W_\bullet H_{\rmB})$ is an $A$-mixed Hodge structure,
	\item $\varphi_{\infty}$ is an $A$-linear involution on $H_{\rmB}$ such that the semi-linear automorphism $\varphi_{\infty}c_{\rmB}$ on $H_{\rmB,\bfC}:=H_{\rmB}\otimes_{A}\bfC$ preserves the Hodge filtration $F^\bullet H_{\rmB,\bfC}$.
	Here, $c_{\rmB}$ denotes the complex conjugation with respect to the real structure $H_{\rmB}\otimes_A\bfR$ of $H_{\rmB,\bfC}$.
\end{itemize}
We set $H_{\dR,\bfR}:=H_{\rmB,\bfC}^{\varphi_{\infty}c_{\rmB}=1}$ and define an automorphism $c_{\dR}$
to be the complex conjugation on $H_{\dR,\bfC}\cong H_{\rmB,\bfC}$ with respect to the real structure $H_{\dR,\bfR}$.
Let $\mathcal{MH}_{A}^+$ denote the category of
$A$-mixed Hodge structures with infinite Frobenius.

Let $H=(H_{\rmB},\varphi_\infty,F^\bullet,W_\bullet)$
be an object of $\mathcal{MH}_{\bfR}^+$ such that $H_{\rmB}=W_0H_{\rmB}$.
Then it is well known that there is a natural isomorphism
\[
\Ext^1_{\mathcal{MH}_{\bfR}^+}(\bfR(0),H)\cong H_{\rmB}^+\backslash H_{\dR,\bfR}/F^0H_{\dR,\bfR},
\]
where $H_{\rmB}^+$ denote the subspace $H_{\rmB}$ on which $\varphi_\infty$
acts as the identity map (\cite[(2.5)]{Nekovar}).
\begin{comment}Explicitly, this isomorphism is given by $1_F-1_W$, where $1_F$ and $1_W$ are Hodge and Frobenius
compatible splittings
\end{comment}
Since the endomorphism $\varphi_\infty-1$ on $H_{\dR,\bfR}$ induces an isomorphism
\[
H_{\rmB}^+\backslash H_{\dR,\bfR}/F^0H_{\dR,\bfR}\xrightarrow{\ \sim\ }H_{\dR,\bfR}^-/(\varphi_\infty-1)F^0H_{\dR,\bfR},
\]
we obtain a natural isomorphism
\begin{equation}
	\label{regulator'}
	r\colon\Ext^1_{\mathcal{MH}_{\bfR}^+}(\bfR(0),H)\xrightarrow{\ \sim\ }H_{\dR,\bfR}^-/(\varphi_\infty-1)F^0H_{\dR,\bfR}.
\end{equation}
Here, $H_{\dR,\bfR}^-$ is the $(-1)$-eigenspace of $\varphi_\infty$ in $H_{\dR,\bfR}$.

We now describe the isomorphism above more explicitly (cf.\ ~\cite[(2.4)]{Nekovar}).
Let 
\[
0\to H\to \widetilde H\to \bfR(0)\to 0
\]
be an exact sequence in $\mathcal{MH}_{\bfR}^+$, and suppose that we are given a splitting
\[
\widetilde H_{\dR,\bfR}=\bfR\oplus H_{\dR,\bfR}
\]
compatible with both the Hodge and weight filtrations.
Let $e\in \widetilde{H}_{\dR,\bfR}$ denote the element corresponding to $1\in \bfR$
under this splitting.
Then the isomorphism $r$ is given by
\begin{equation}
	r([\widetilde H])=(\varphi_\infty-1)(e)\modulo\ (\varphi_\infty-1)F^0H_{\dR,\bfR}.
\end{equation}
Therefore, in this setting, it suffices to compute the matrix coefficients of $\varphi_\infty$ in order to determine $r([M])$.

For later use, we record two important examples.
Let $X$ be a projective smooth curve defined over $\bfR$. Then it is well known that $H=\rmH^1(X(\bfC),\bfR)$ carries a natural structure of an object of $\mathcal{MH}_{\bfR}^+$, with $H_{\dR,\bfR}=\rmH^1_{\dR}(X/\bfR)$.
Let us take a basis $\calB_0\subset H_{\dR,\bfR}$ as before, and write $\calB_h=\{[\omega_1],\dots,[\omega_n]\}$.
Let $V$ be an object of $\mathcal{MH}_{\bfR}^+$, pure of weight $0$.
We write $V^{\pm}$ for the $\pm$-eigenspaces of $\varphi_{\infty}$ on $V_{\dR,\bfR}$.
\begin{example}
	\label{ex4-0}Let us consider the object $H(2)\otimes V$ in $\mathcal{MH}_{\bfR}^+$.
	Since $F^0(H(2)\otimes V)=0$, we have a natural isomorphism
	\[
r\colon \Ext^1_{\mathcal{MH}_{\bfR}^+}(\bfR(0),H(2)\otimes V)	\isom  H_{\dR,\bfR}^{-}\otimes V^+\oplus H_{\dR,\bfR}^{+}\otimes V^-.
	\]
	For $\varepsilon\in \{\pm\}$ and $v\in V^{\varepsilon}$, we define $e_i\otimes v\in \Ext^1_{\mathcal{MH}_{\bfR}^+}(\bfR(0),H(2)\otimes V)$ to be the image of
	\[
	([\omega_i]-\varepsilon [\overline{\omega_i}])\otimes v
	\] 
	under the isomorphism above.
	Then, for a basis $\{v_l\}_l$ of $V=V^{+}\oplus V^{-}$,
	the set $\{e_i\otimes v_l\}_{i,l}$ forms a basis of $\Ext^1_{\mathcal{MH}_{\bfR}^+}(\bfR(0),H(2)\otimes V)$.
\end{example}
\begin{example}
	\label{ex4-1}
	Next, let us consider the mixed Hodge structure with infinite Frobenius $H^{\otimes 2}(2)\otimes V$.
	Since $\calB_h$ is a basis for $F^1H_{\dR,\bfR}$, a basis of $F^0H^{\otimes 2}(2)_{\dR,\bfR}=F^2H^{\otimes 2}_{\dR,\bfR}$
	is given by $\{[\omega_i]\otimes[\omega_j]\}_{1\leq i,j\leq n}$.
	Moreover, since the set
	\[
	\{[\omega_i]\otimes[\omega_j]+\varepsilon[\overline{\omega}_i]\otimes[\overline{\omega_j}]\ |\ 1\leq i,j\leq n\}\coprod\{[\omega_i]\otimes[\overline{\omega_j}]+\varepsilon [\overline{\omega_i}]\otimes[{\omega_i}]\ |\ 1\leq i,j\leq n\}
	\]
	is a basis for $\left(H^{\otimes 2}_{\dR,\bfR}\right)^{\varepsilon}$,
	we have a natural isomorphism
	\begin{multline}
		\label{eq19}
		r\colon \Ext^1_{\mathcal{MH}_\bfR^+}(\bfR(0),H^{\otimes 2}(2)\otimes V)\\
		\isom \bigoplus_{\varepsilon\in\{+,-\}}\left(\bigoplus_{1\leq i, j\leq n}\bfR([\omega_i]\otimes[\overline{\omega_j}]-\varepsilon[\overline{\omega_i}]\otimes[{\omega_j}])\right)\otimes_{\bfR}V^{\varepsilon}.
	\end{multline}
	For $v\in V^{\varepsilon}$, we define $e_{ij}\otimes v\in \Ext^1_{\mathcal{MH}_\bfR^+}(\bfR(0),H^{\otimes 2}(2)\otimes V)$, for $[\omega_i],[\omega_j]\in\calB_h$,
	to be the image of $([\overline{\omega_i}]\otimes[{\omega_j}]-\varepsilon[{\omega_i}]\otimes[\overline{\omega_j}])\otimes v$ under the isomorphism above.
	Then, similarly to Example~\ref{ex4-0}, these elements span $\Ext^1_{\mathcal{MH}_\bfR^+}(\bfR(0),H^{\otimes 2}(2)\otimes V)$.
	
	Let $v_1,\dots ,v_m$ be a basis for $V_{\dR,\bfR}$ with each $v_i\in V^{\varepsilon}$ for some $\epsilon\in \{+,-\}$.
	Suppose that we are given an extension
	\[
	0\to H^{\otimes 2}(2)\otimes V\to \widetilde H\to \bfR(0)\to 0
	\]
	in $\mathcal{MH}_{\bfR}^+$, equipped with a de Rham splitting $\widetilde H_{\dR,\bfR}=\bfR\oplus H^{\otimes 2}(2)_{\dR,\bfR}\otimes_{\bfR} V_{\dR,\bfR}$.
	Then the extension class $[\widetilde H]\in \Ext^1_{\mathcal{MH}_\bfR^+}(\bfR(0),H^{\otimes 2}(2)\otimes V)$
	is given by $\sum_{1\leq i, j\leq n,1\leq k\leq m}a_{ij,k}e_{ij}\otimes v_k$, where
	the real numbers $a_{ij,k}$ are determined by the equation
	\begin{equation}
		\label{eq20}
		\varphi_{\infty}(1)=1+\sum_{1\leq i,j\leq n,1\leq k\leq m}a_{ij,k}[\omega_i]\otimes[\overline{\omega_j}]\otimes v_k+\cdots.
	\end{equation}
\end{example}

We have a natural isomorphism
\[
\Ext^1_{\mathcal{MH}_{\bfQ}^+}(H,H')\isom\Ext^1_{\mathcal{MH}_{\bfQ}^+}(\bfQ(0),H^\vee\otimes H')
\]
between these two extension groups. This isomorphism may be described explicitly as follows:
Let \[
E\colon 0\to H'\to \widetilde H\to H\to 0
\] be an extension in $\mathcal{MH}_{\bfQ}^+$.
Then the corresponding extension is obtained by pulling back the extension $H^\vee\otimes E$
along the identity map $\bfQ(0)\to H^\vee\otimes H=\underline{\End}(H)$.
We write $\reg_{\bfR,H,H'}$ as the composition of homomorphisms
\begin{multline}
	\label{regulator}
	\reg_{\bfR, H,H'}\colon \Ext^1_{\mathcal{MH}_{\bfQ}^+}(H,H')\isom\Ext^1_{\mathcal{MH}_{\bfQ}^+}(\bfQ(0),H^\vee\otimes H')\\
	\to \Ext^1_{\mathcal{MH}_{\bfR}^+}(\bfR(0),H^\vee\otimes H')
\end{multline}

\subsection{The regulator formulas}\label{regformula}
Let $F$ be a subfield of $\bfR$.
Let $X$ be a projective smooth curve over $F$,
and let $D$ be a non-empty zero-dimensional closed subset of $X$.
Set $Y:=X\setminus D$, and fix an $F$-rational base point $b$ of $Y$.
For each smooth variety $S$ over $F$, we equip $\rmH^i(S(\bfC),\bfQ)$ with Deligne's mixed Hodge structure.
Write $\rmH^i(S)$ for this $\bfQ$-mixed Hodge structure with infinite Frobenius, and denote its dual by $\rmH_i(S)$. 

We denote by $I_b$ the augmentation ideal of the group ring $\bfQ[\pi_1^{\top}(Y,b)]$.
As explained in Section~\ref{MHS} and Section~\ref{BWfuct},
$I_b/I_b^n$ carries a natural $\bfQ$-mixed Hodge structure with infinite Frobenius such that
$I_b/I_b^{n+1}\ontomap I_b/I_b^{n}$ is a morphism in $\mathcal{MH}_{\bfQ}^+$, and such that there is a natural isomorphism
\[
I_b^n/I_b^{n+1}\cong \rmH_1(Y)^{\otimes n}
\]
of $\bfQ$-mixed Hodge structures with infinite Frobenius.
Therefore, we obtain an extension
\[
0\to \rmH_1(Y)^{\otimes 2}\to I_b/I_b^3\to \rmH_1(Y)\to 0
\]
in $\mathcal{MH}_{\bfQ}^+$. We denote this extension class by \[
[I_b/I_b^3]\in \Ext^1_{\mathcal{MH}_{\bfQ}^+}( \rmH_1(Y), \rmH_1(Y)^{\otimes 2}).
\]

To simplify the computation of $\reg_{\bfR,\rmH_1(Y),\rmH_1(Y)^{\otimes 2}} ([I_b/I_b^3])$,
we assume the following condition on $D$:
\begin{center}
	(\textbf{Tor})  The divisor $(d)-(d')$ is torsion in $\mathrm{Jac}(X)$ for all $d,d'\in D(\bfC)$.
\end{center}
Here, $\mathrm{Jac}(X)$ is the Jacobian variety of $X$, and we regard $D$ as a reduced closed subscheme of $X$.
By the assumption above, we have a natural splitting
\[
\rmH_1(Y)=\rmH_1(X)\oplus V(1)
\]
in $\mathcal{MH}_{\bfQ}^+$, where $V$ is the weight zero Hodge structure $\rmH^0(D(\bfC),\bfQ)/\text{diagonals}$.
Therefore, we obtain the following natural isomorphisms in $\mathcal{MH}_{\bfQ}^+$:
\begin{multline}\label{eq28}
	\rmH^1(Y)\otimes \rmH_1(Y)^{\otimes 2}\cong \bigl( \rmH^1(X)\otimes \rmH_1(Y)^{\otimes 2}\bigr)\oplus \bigl(V(1)\otimes \rmH_1(Y)^{\otimes 2}\bigr)\\
	\cong  \rmH^1(X)^{\otimes 3}(2)\oplus\bigl(\rmH^1(X)\otimes \rmH^1(X)(1)\otimes V(1)\bigr)\oplus \bigl(\rmH^1(X)\otimes V(1) \otimes \rmH^1(X)(1)\bigr)\\
	\bigl( \rmH^1(X)\otimes V^{\otimes 2}(2)\bigr)
	\oplus  \bigl(V(1)\otimes \rmH_1(Y)^{\otimes 2}\bigr)
	=:H_1\oplus H_2\oplus H_2'\oplus H_3\oplus H_4.
\end{multline}
According to the splitting above, $\reg_{\bfR,\rmH_1(Y),\rmH_1(Y)^{\otimes 2}}$ also decomposes as
\[
\reg_{\bfR,\rmH_1(Y),\rmH_1(Y)^{\otimes 2}}=\reg_1\oplus\reg_2\oplus \reg_2'\oplus \reg_3\oplus \reg_4,
\]
where $\reg_i$ is a homomorphism to $\Ext^1_{\mathcal{MH}_\bfR^+}(\bfR(0),H_{i,\bfR})$.

Note that $\reg_1=0$ because the target group is zero.
Moreover, if either $\reg_2([I_b/I_b^3])$ or $\reg_2'([I_b/I_b^3])$ can be computed,
then the other one also can be determined by the symmetry of matrix coefficients.

In the remainder of this paper, we compute $\reg_2([I_b/I_b^3])$ and $\reg_3([I_b/I_b^3])$
based on the computation of  matrix coefficient of the second and third cases, respectively.
Finally, we remark that the computation of $\reg_4([I_b/I_b^3])$ corresponds to the remaining subcases in the first and third cases in the previous section.
Thus, the calculation of $\reg_4([I_b/I_b^3])$ remains incomplete.

Take a basis\[
\calB=\calB_h\coprod\calB_{ah}\coprod\calB_e^+\coprod\calB_e^-
\]
of $\rmH^1_{\dR}(Y/\bfR)$, as in Section~\ref{theformula}.
That is, $\calB_h=\{[\omega_1],\dots,[\omega_n]\}$ is an orthogonal basis of $H^0(X,\Omega_{X\times_{\spec(F)}\spec(\bfR)}^1)$
with respect to the pairing
\[
(\alpha,\beta):=\frac{1}{2\pi\sqrt{-1}}\int_{X(\bfC)}\alpha\wedge\overline{\beta},
\]
and $\calB_{ah}$ is defined as $\{[\overline{\omega_i}]\ |\ i=1,\dots,n\}$.
For each $[\omega]\in \calB_e$, choose a smooth function $\calE_{\omega}$ on $Y(\bfC)$, with at worst logarithmic singularities along
$D$ satisfying the equations (\ref{eq29}). If $b$ is a tangential base point, we normalize this function by the condition (\ref{eq18}).

Note that $V_{\dR,\bfR}$ can be identified with the dual $\bfR$-vector space spanned by $\calB_e$. Therefore, we have a natural isomorphism
\[
V_{\dR,\bfR}^{\pm}=\bigoplus_{[\omega]\in \calB_e^{\pm}}\bfR[\omega]^\vee.
\]
For $1\leq i, j\leq n,\ \varepsilon\in \{+,-\}$, and $\omega\in \calB_e$, let $e_{ij}\otimes[\omega]^\vee$ denote the image of $([\overline{\omega_i}]\otimes[{\omega_j}]-\varepsilon[{\omega_i}]\otimes[\overline{\omega_j}])\otimes[\omega]^\vee$
in the extension group $\Ext^1_{\mathcal{MH}_\bfR^+}(\bfR(0),H^1(X)^{\otimes 2}(2)\otimes V)$ (see Example~\ref{ex4-1}).
\begin{theorem}
	\label{thm4-2}Suppose that the condition (\textbf{Tor}) holds for $D$. If $b$ lies over an $F$-rational point of $Y$, then the following identity holds$:$
	\begin{multline}
		\reg_2([I_b/I_b^3])=\sum_{1\leq i,j\leq n,\ i\neq j,\ [\omega]\in \calB_e}\frac{(\calE_{\omega}\omega_j,\omega_i)}{(\omega_i,\omega_i)(\omega_j,\omega_j)}e_{ji}\otimes[\omega]^\vee\\
		+\sum_{1\leq i\leq n,[\omega]\in \calB_e}\frac{((\calE_{\omega}-\calE_{\omega}(b))\omega_i,\omega_i)}{(\omega_i,\omega_i)^2}e_{ii}\otimes[\omega]^\vee.
	\end{multline}
	If $b$ is an $F$-rational tangential base point of $Y$, then the following identity holds$:$
	\begin{equation}
		\reg_2([I_b/I_b^3])=\sum_{1\leq i,j\leq n,\  [\omega]\in \calB_e}\frac{(\calE_{\omega}\omega_j,\omega_i)}{(\omega_i,\omega_i)(\omega_j,\omega_j)}e_{ji}\otimes[\omega]^\vee.
	\end{equation}
\end{theorem}
\begin{proof}
	To compute $\reg_2([I_b/I_b^3])$, we take a suitable one-form
	$\Omega$ as in Theorem~\ref{thm3-1}.
	By the trivialization $t_{\Omega,b}^b$ arising from $\Omega$, we have a natural splitting
	\[
	(I_b/I_b^3)_{\dR,\bfR}\cong \rmH_1^{\dR}(Y/\bfR)\oplus \rmH_1^{\dR}(Y/\bfR)^{\otimes 2}
	\]
	compatible with the Hodge and weight filtrations.
	According to Theorem~\ref{thm3-1}, when $b$ lies over an $F$-rational point of $Y$, we have
	\begin{multline}
		\varphi_{\infty}([\overline{\omega_i}]^\vee)=[{\omega_i}]^\vee+\sum_{j\neq i,\ [\omega]\in \calB_e}\frac{(\calE_\omega\omega_j,\omega_i)}{(\omega_i,\omega_i)}[\omega_j]^\vee\otimes[\omega]^\vee\\
		+\sum_{[\omega]\in\calB_e}\frac{((\calE_\omega-\calE_\omega(b))\omega_i,\omega_i)}{(\omega_i,\omega_i)}[\omega_i]^\vee\otimes[\omega]^\vee+\cdots.
	\end{multline}
	Therefore, the following identity on $(I_b/I_b^3)_{\dR,\bfR}\otimes_{\bfR} \rmH^1_{\dR}(X/\bfR)$ holds:
	\begin{multline}
		\varphi_{\infty}([\overline{\omega_i}]^\vee\otimes[\overline{\omega_i}])=[{\omega_i}]^\vee\otimes[{\omega}_i]+\sum_{j\neq i,\ [\omega]\in\calB_e}\frac{(\calE_\omega\omega_j,\omega_i)}{(\omega_i,\omega_i)}[\omega_j]^\vee\otimes[\omega]^\vee\otimes[{\omega_i}]\\
		+\sum_{[\omega]\in\calB_e}\frac{((\calE_\omega-\calE_\omega(b))\omega_i,\omega_i)}{(\omega_i,\omega_i)}[\omega_i]^\vee\otimes[\omega]^\vee\otimes[{\omega}_i]+\cdots.
	\end{multline}
	Note that the element $[\omega_j]^\vee$ is sent to $\frac{1}{(\omega_j,\omega_j)}[\overline{\omega_j}]$ under the inverse of the natural isomorphism
	\[
	\rmH^1_{\dR}(X/\bfC)(1)=\rmH^1_{\dR}(X/\bfC)\isom\rmH_1^{\dR}(X/\bfC);\quad \alpha\mapsto \frac{1}{2\pi\sqrt{-1}}\int_{X(\bfC)}\alpha\wedge(-).
	\]
	Since our extension is obtained by pulling back the above mixed Hodge structure via
	\[
	\bfR(0)\to \rmH_1^{\dR}(X/\bfR)\otimes_{\bfR}\rmH^1_{\dR}(X/\bfR);\quad 1\mapsto\sum_{\omega\in \calB_0}[\omega]^\vee\otimes[\omega],
	\]
	we obtain the first assertion of the theorem.
	The tangential base point case can be computed in the same way by applying Theorem~\ref{thm3-2} instead of Theorem~\ref{thm3-1}.
\end{proof}
Next, we calculate $\reg_3([I_b/I_b^3])$.
For each $[\omega],[\omega']\in  \calB_e$, we take a closed one-form
\[
\calF_{\omega,\omega'}
\]
defined by the equation (\ref{eq27}).
Moreover, for simplicity, we assume the following condition:
\begin{center}
	(\textbf{Orth}) The basis $\calB_e$ is orthogonal to $\calB_0$ with respect to the
	pairing $(\ ,\ )$.
\end{center}
We define the element \[
e_i\otimes[\omega]^\vee\otimes[\omega']^\vee\in \Ext^1(\bfR(0),H^1(X)(2)\otimes V^{\otimes 2})
\]
to be the image of $([\omega_i]-\varepsilon(\omega,\omega')[\overline{\omega_i}])\otimes [\omega]^\vee\otimes[\omega']^\vee$ (see Example~\ref{ex4-0}).
\begin{theorem}
	\label{thm4-3}Under the condition (\textbf{Orth}), we have the following identity$:$
	\[
	\reg_3([I_b/I_b^3])=-\frac{1}{2}\sum_{1\leq i\leq n,\ [\omega],[\omega']\in \calB_e}\frac{(\calF_{\omega,\omega'},\omega_i)}{(\omega_i,\omega_i)}e_i\otimes [\omega]^\vee\otimes[\omega']^\vee.
	\]
\end{theorem}
\begin{proof}The proof proceeds in a very similar manner to that of Theorem~\ref{thm4-2}.
	Instead of Theorem~\ref{thm3-1}, we use Theorem~\ref{thm3-3} here.
\end{proof}
\begin{remark}
	It is well known that when $b$ lies over an $F$-rational point of $Y$, the quotient $I_b/I_b^3$ is the mixed Hodge realization of a certain mixed motive over $F$
	in the sense of Nori (\cite[Theorem 16.4]{HuM17}).
	Therefore, Theorem~\ref{thm4-2} and Theorem~\ref{thm4-3} may be regarded as regulator formulas
	for certain Nori motives over $F$.
\end{remark}
\begin{comment}
\begin{remark}
	Theorem~\ref{thm4-2} and Theorem~\ref{thm4-3} do not contribute to solve the classical Beilinson conjecture.
	On the contrary, similar result was already appeared in Beilinson's original paper~\cite{Beilinson} (see also~\cite[Section 3]{De-G05}).
	Nevertheless, our method seems to have good points in two respects. First, we do not need the condition (\textbf{Tor}) for computations of the infinite Frobenius actions, and second, it can be easily generalizable to the relatively unipotent cases.
\end{remark}
\end{comment}

%%%%%%%%%%%%%%%%%%%%%
%%%%%%%%%%%%%%%%%%%%%
%%%%%%%%%%%%%%%%%%%%%%
\section{The modular curve $Y_0(N)$}
For a positive integer $N$, $[\Gamma_0(N)]$ denotes the moduli problem classifying elliptic curves with a cyclic subgroup of order $N$. According to~\cite{KM84}, this moduli problem has a coarse moduli schemes $Y_0(N)_{\bfZ}$ over $\bfZ$. Let $Y_0(N)$ denote the base change of this scheme to $\bfQ$,
which is a smooth affine curve over $\bfQ$.
In this section, we focus on the modular curve $Y_0(N)$.

Let $b$ be an $\bfR$-rational base point of $Y_0(N)$, and let $I_b$
denote the augmentation ideal of $\bfQ[\pi_1(Y_0(N)(\bfC),b)]$.
In this section, we compute $\reg_i([I_b/I_b^3])$ for $i=2,3$.
Fix a positive integer $N$ greater than two.
We write $X_0(N)$ as the smooth compactification of $Y_0(N)$ over $\bfQ$.
By Drinfeld--Manin's theorem, the pair $(X_0(N), Y_0(N))$ satisfies the condition (\textbf{Tor}).

Let $M_2(\Gamma_0(N))$ (resp.\ $S_2(\Gamma_0(N))$) denote the space of modular forms (resp.\ cusp forms) of weight two and level $N$ with the trivial character. 
For a modular form $f\in M_2(\Gamma_0(N))$, we write $\omega_f$ for the differential form on the upper half-plane $\frakH$ defined by
\begin{equation}
	\label{eq33}
	\omega_f=2\pi\sqrt{-1}f(\tau)d\tau,
\end{equation}
where $\tau$ is the standard coordinate on $\frakH$.
Via the complex uniformization
\[
\Gamma_0(N)\backslash	\frakH\isom Y_0(N)(\bfC);\quad \tau\mapsto \left[\bfC/(\bfZ \tau+\bfZ),\left\la \frac{1}{N}\right\ra\right],
\]
we regard $\omega_f$ as a differential form on $Y_0(N)$.

\subsection{Eisenstein series}\label{eis}
Let $G=\GL_2/\bfZ$ and let $B$ be the Borel subgroup of $G$ consisting of upper-triangular matrices.
We denote by $U$ the unipotent radical of $B$.
Let $\bfA_{\rmf}$ (resp.\ $\bfA$) denote the ring of finite adeles (resp.\ adeles).
Let \[
\delta_{\rmf}\colon B(\bfA_{\rmf})\to \bfR_+^\times
\] 
be the character defined by
\[
\delta_{\rmf}\left(\begin{bmatrix}
	a&b\\0&d
\end{bmatrix}\right)=\left|\frac{a}{d}\right|_{\rmf},
\]
where $|\ |_{\rmf}$ denotes the usual norm on $\bfA_{\rmf}$.

Let $K_{\rmf}=K_0(N)$ be the compact open subgroup of $G(\bfA_{\rmf})$ defined by
\[
K_{\rmf}=\left\{\begin{bmatrix}
	a&b\\
	c&d\end{bmatrix}\in G(\widehat{\bfZ})\ \middle |\ c\in N\widehat{\bfZ}\right\}.
\]
Let $K_\infty$ be $\rmO(2)$.
Then we have well-known uniformizations
\[
G(\bfQ)\backslash G(\bfA)/Z(\GL_2(\bfR))K_\infty K_{\rmf}\isom Y_0(N)(\bfC)
\]
and
\[
U(\widehat{\bfZ})\backslash G(\widehat{\bfZ})/K_{\rmf}\isom (X_0(N)\setminus Y_0(N))(\bfC)
\]
(see~\cite[3.0.1]{SS}).
Put \[
\Cusp:=U(\widehat{\bfZ})\backslash G(\widehat{\bfZ})/K_{\rmf}
\]
and we call
an element of $\Cusp $ a \emph{cusp}.
The cusp represented by $g\in G(\widehat{\bfZ})$ is denoted by $[g]$.
We sometimes regard $[g]$ as an open subset of $G(\widehat{\bfZ})$.
The standard cusp $\infty$ is the cusp represented by $2\times 2$ identity matrix.

For a cusp $c=[g]\neq \infty$, define a locally constant function $\phi_c\colon G(\widehat{\bfZ})\to \bfQ$ by
\[
\phi_c=\mathrm{char}_{\infty}- \#\left( U(\widehat{\bfZ})gK_{\rmf}/K_{\rmf}\right)^{-1}\mathrm{char}_{[g]},
\]
where $\mathrm{char}_C$ denotes the characteristic function associated to an open subset $C$ of $G(\widehat{\bfZ})$.
Let $B(\bfQ)^+$ be the subgroup of $B(\bfQ)$ consisting of elements with positive determinant.
Then, by the Iwasawa decomposition, we have \[
G(\bfA_{\rmf})=B(\bfQ)^+G(\widehat{\bfZ}).
\]
For each complex number $s$, define the function $\widehat{\phi}_{c,s}\colon G(\bfA_{\rmf})\to \bfC$ by
\[
\widehat{\phi}_{c,s}(bk)=\delta_{\rmf}(b)^{2s}\phi_{c}(k),\quad b\in B(\bfQ)^+,\quad k\in G(\widehat{\bfZ}).
\]
This map is well-defined. Indeed,
if $bk=b'k'$ for $b,b'\in B(\bfQ)^+,\ k,k'\in G(\widehat{\bfZ})$, we have $b^{-1}b',\ k^{-1}k'\in \pm U(\bfZ)$, because $B(\bfQ)^+\cap G(\widehat{\bfZ})=\pm U(\bfZ)$.
Therefore, we have \[
\delta_{\rmf}(b)=\delta_{\rmf}(b'),\quad \phi_c(k)=\phi_c(k').
\]

The \emph{Eisenstein series $\calE_{c}(-,s)\colon G(\bfA)\to \bfC$ associated with $\phi_{c}$ of weight zero} is defined by the Poincar\'e series
\[
\calE_{c}(h,s)=-4\pi\sum_{\gamma\in B(\bfQ)^+\backslash G(\bfQ)}\widehat{\phi}_{c,s}(\gamma h_{\rmf})I(\gamma h_\infty)^s
\]
(see~\cite[(3.1.4)]{SS}), where $h=(h_{\infty},h_{\rmf})$ is the decomposition of $h$ corresponding to $G(\bfA)=G(\bfR)\times G(\bfA_{\rmf})$, and $I(h_\infty)$ is defined by the equation
\[
I(h_\infty)=\frac{1}{2}
(	\mathrm{Im}(h_\infty\sqrt{-1})+\mathrm{sgn}(\mathrm{Im}(h_\infty\sqrt{-1}))\mathrm{Im}(h_\infty\sqrt{-1})).
\]
This series converges absolutely if $\mathrm{Re}(s)>2$, and admits an analytic continuation to the entire $s$-plane
by the classical Hecke trick (\cite[(3.1.7)]{SS}; cf.\ ~\cite[Section 7.2]{Miyake}).
For any fixed $s$, we regard this as a function on the upper-half plane $\frakH$ and write $\calE_c(\tau,s)$ for this function,
where $\tau$ is the standard coordinate of $\frakH$.
\begin{definition}
	For each $c\in \Cusp$, we define the smooth $(1,0)$-form $\eta_c$ on $\frakH$ by\[
	\eta_c=\partial_{\tau}\calE_c(\tau,1).
	\]
\end{definition}
Note that $\eta_c$ is holomorphic, since $\calE_c(\tau,1)$ is harmonic (\cite[(3.1.7)]{SS}).
We define the modular form $E_c$ of weight two and level $\Gamma_0(N)$ by the equation $\eta_c=\omega_{E_c}$.
By definition, we have $d\calE_c(\tau,1)=2\mathrm{Re}(\eta_c)$.
\begin{proposition}\label{prop5-0}
	The set $\{E_c\}_{c\in \Cusp,\ c\neq \infty}$ is a basis of the space $M_2(\Gamma_0(N))/S_2(\Gamma_0(N))$ of Eisenstein series.
\end{proposition}
\begin{proof}Let $c'$ be a cusp. Then, it is easily checked that $\lim_{\tau\to c'}E_c(\tau)$ is non-vanishing if and only if $c'$ is $c$ or $\infty$ (~\cite[(3.1.7)]{SS}; cf.\ ~\cite[Exercise 4.2.3, Section 4.6]{DS}).
	Therefore, the set $\{E_c\}_{c\in \Cusp,\ c\neq \infty}$ is linearly independent over $\bfC$.
	Thus, we obtain the conclusion by dimension counting.
\end{proof}

\subsection{Rankin's trick}
For $f\in M_2(\Gamma_0(N))$, we write $\varphi_f$ for the corresponding automorphic form on $G(\bfA)$. Explicitly, we have
\[
\varphi_f(\gamma z g_\infty k)=f(g_\infty\sqrt{-1})j(g_\infty,\sqrt{-1})^2,\quad \gamma\in G(\bfQ),\ z\in Z(G(\bfR)),\ g_\infty\in \SL_2(\bfR),\ k\in K_{\rmf}.
\]Let
\[
\psi\colon \bfA\to \bfC^\times
\]
be the additive character on $\bfA$ defined by
\[
\psi_p\colon \bfQ_p\ontomap \bfQ_p/\bfZ_p\xrightarrow{\ \exp(2\pi\sqrt{-1}(-))\ }\bfC^\times
\]
and
\[
\psi_\infty(a):=\exp(2\pi\sqrt{-1}a),\quad a\in \bfR.
\]
The symbol $W_f=W_{f,\infty}\otimes W_f^{\rmf}\colon G(\bfA)\to \bfC$ denotes the Whittaker function associated with $\varphi_f$
and $\psi$.

For $c\in \Cusp$, and for $f,g\in M_2(\Gamma_0(N))$ such that both are normalized and one of them is cuspidal,
we define the function $I(s;f,g;c)$ in terms of the complex parameter $s$ by
\[
I(s;f,g;c):=\int_{\bfA^{\times}_{\rmf}\times G(\widehat{\bfZ})}\phi_c(k)W_f^{\rmf}\left(\begin{bmatrix}
	a&0\\
	0&1
\end{bmatrix}k\right)\overline{W_g^{\rmf}}\left(\begin{bmatrix}
	a&0\\
	0&1
\end{bmatrix}k\right)|a|_{\rmf}^{s-1}d^\times adk,
\]
where $dk $ is the Haar measure of $G(\bfA_{\rmf})$normalized so that the total volume of $G(\widehat{\bfZ})$ is one.
This integral converges if $\mathrm{Re}(s)>1$.
\begin{proposition}[{Rankin's trick,~\cite[Proposition 5.1.0]{SS}; cf.\ ~\cite[Proposition 3.8.2]{Bump}}]For each $f$, $g$, and $c$ as above, we have
	\begin{equation}
		(\calE_c(\tau,s)\omega_f,\omega_g)=-2\pi\sqrt{-1}\frac{\Gamma(s+1)}{(4\pi)^{s-1}}|G(\widehat{\bfZ}):K_0(N)|I(s;f,g;c)\label{eq30}
	\end{equation}
	for all complex number $s$ such that $\mathrm{Re}(s)>1$.
	\label{prop5-1}
\end{proposition}
The symbols $\calW(\pi_{f,p},\psi_p)$ and $\calW(\pi_{g,p},\psi_p)$ denote the Whittaker models of $\pi_{f,p}$ and $\pi_{g,p}$ associated with $\psi_p$, respectively.
Then, by the tensor product theorem~\cite[Theorem 3.3.3]{Bump}, we have a partial tensor product decomposition
\[
W_f=\otimes'_{v\ \nmid\ N}W_{f,v}\otimes W_{f,N},\quad W_g=\otimes'_{v\ \nmid\ N}W_{g,v}\otimes W_{g,N}
\]
of Whittaker functions, where $W_{f,p}$ and $W_{g,p}$ are the unique normalized spherical vectors in $\calW(\pi_{f,p},\psi_p)$ and $\calW(\pi_{g,p},\psi_p)$, respectively.
Let $\phi_{c,N}\colon \prod_{p|N}G(\bfZ_p)\to \bfC$ be the locally constant function satisfying $\phi_c=\mathrm{char}_{\prod_{p\nmid N}G(\bfZ_p)}\times \phi_{c,N}$.
Put
\begin{equation*}
	I_N(s;f,g;c):=\int_{\prod_{p|N}(\bfZ_p^{\times}\times G(\bfZ_p))}\phi_{c,N}(k)W_{f,N}\left(\begin{bmatrix}
		a&0\\
		0&1
	\end{bmatrix}k\right)\overline{W_{g,N}}\left(\begin{bmatrix}
		a&0\\
		0&1
	\end{bmatrix}k\right)|a|_N^{s-1}d^\times adk,
\end{equation*}
where $|a|_N:=\prod_{p|N}|a_p|_p$.

For each prime number $p$ coprime to $N$, $L_p(s,\pi_f\times \pi_g)$ is given by
\begin{equation}
	\label{eq32}
	L_p(s,\pi_f\times \pi_g)=\frac{1}{(1-\alpha\beta p^{-s})(1-\alpha'\beta p^{-s})(1-\alpha\beta' p^{-s})(1-\alpha'\beta' p^{-s})},
\end{equation}
where $\{\alpha,\alpha'\}$ and $\{\beta,\beta'\}$ are the Satake parameters of $\pi_{f,p}$ and $\pi_{g,p}$, respectively.
Then, $L^{(N)}(s,\pi_f\times\pi_g)$ is defined by
\begin{equation}
	\label{eq31}
	L^{(N)}(s,\pi_f\times\pi_g)=\prod_{p\nmid N}L_p(s,\pi_f\times\pi_g).
\end{equation}
It is easily checked that this Euler product converges if $\mathrm{Re}(s)>1$. Moreover, according to~\cite[Proposition 3.8.4]{Bump}, $L^{(N)}(s,\pi_f\times \pi_g)$ has a meromorphic continuation to $\bfC$. It is holomorphic if and only if $\pi_f\not\cong\breve{\pi}_g$, and has a simple pole at $s=1$ if $\pi_f\cong\breve{\pi}_g$ (\cite[Proposition 3.8.5, p.375]{Bump}).
The following proposition is well known.
\begin{proposition}
	\label{prop5-2}
	We have the identity
	\[
	I(s;f,g;c)=\zeta^{(N)}(2s)^{-1}L^{(N)}(s,\pi_f\times\pi_g)I_N(s;f,g,c),
	\]
	where $\zeta^{(N)}(s):=\prod_{p\nmid N}(1-p^{-s})^{-1}$.
\end{proposition}
\begin{remark}
	Suppose that $f$ and $g$ are Hecke eigenforms. Let $\ell$ be a prime number, and let $V_f$ and $V_g$ be the $\ell$-adic Galois representations pure of weight one
	associated with $f$ and $g$, respectively.
	Let $
	L_p(s,V_f\otimes V_g)$ denote the $p$-Euler factor of the Hasse--Weil L-function of $V_f\otimes V_g$.
	Then, for $p\nmid N\ell$, we have
	\[
	L_p(s+1,V_f\otimes V_g)=L_p(s,\pi_f\otimes \pi_g).
	\]
\end{remark}
\begin{comment}
	\begin{proof}
		It suffices to show the equation
		\begin{equation}\label{eq35}
			\int_{\bfQ_p^\times}W_{f,p}\left(\begin{bmatrix}
				a&0\\
				0&1\end{bmatrix}\right)\overline{W}_{g,p}\left(\begin{bmatrix}
				a&0\\
				0&1\end{bmatrix}\right)|a|_p^{s-1}d^\times a=L_p(s,\pi_{f,p}\times \pi_{g,p})(1-p^{-s})
		\end{equation}
		holds for all $p\nmid N$. The right-hand side above is computed as follows:
		\[
		\text{r.h.s. of }(\ref{eq35})=\sum_{n=0}^\infty W_{f,p}\left(\begin{bmatrix}
			p^n&0\\
			0&1\end{bmatrix}\right)\overline{W}_{g,p}\left(\begin{bmatrix}
			p^n&0\\
			0&1\end{bmatrix}\right)p^{-n(s-1)}.
		\]
		Since the level of $g$ is $\Gamma_0(N)$, its Satake parameter at $p$ is stable under the complex conjugation.
		Therefore, the following equations hold:
		\begin{equation*}
			\begin{split}
				\sum_{n=0}^\infty& W_{f,p}\left(\begin{bmatrix}
					p^n&0\\
					0&1\end{bmatrix}\right)p^{-n(s-1/2)}=\frac{1}{(1-\alpha p^{-s})(1-\alpha' p^{-s})},\\
				\sum_{n=0}^\infty &\overline{W}_{g,p}\left(\begin{bmatrix}
					p^n&0\\
					0&1\end{bmatrix}\right)p^{-n(s-1/2)}=\frac{1}{(1-\beta p^{-s})(1-\beta' p^{-s})}.
			\end{split}
		\end{equation*}
		Thus, we have the conclusion of the proposition by~\cite[Lemma 1]{Shimura2}.
	\end{proof}
\end{comment}
\subsection{Regulator formula}
The Hecke operators are normal with respect to the Petersson inner product (\cite[Theorem 3.41]{Shimura1}).
Therefore, the set
\[
\calB_h=\{\omega_f\ |\ f\in S_2(\Gamma_0(N)),\ f\text{ is a normalized cuspidal Hecke eigenform}\}
\]
is a basis of $\rmH^0(X_0(N),\Omega^1_{X_0(N)/\bfC})$.
We often identify $\calB_h$ with the set of normalized cuspidal Hecke eigenforms.
Note that $\omega_f$ is defined over the Hecke field of $f$,
which is a totally real number field. Therefore, in particular, $\calB_h$ is a basis of $\rmH^0(X_0(N),\Omega^1_{X_0(N)/\bfR})$.
Put 
\[
\calB_{ah}:=\{\overline{\omega_f}\ |\ \omega_f\in \calB_h\}.
\]
Then, $\calB_0:=\calB_h\coprod \calB_{ah}$ forms a basis of $\rmH^1_{\dR}(X/\bfR)$.
We define $\calB_e$ by
\[
\calB_e=\{\eta_c\ |\ c\in \Cusp\setminus\{\infty\}\}.
\]
According to Proposition~\ref{prop5-0}, the set $\calB:=\calB_0\coprod \calB_e$ is a basis of $\rmH^1_{\dR}(Y_0(N)/\bfR)$
satisfying the condition (\textbf{Orth}).

As in Subsection~\ref{regformula}, we define the weight zero Hodge structure $V$ by
\[
V(1)=\Ker(\rmH_1(Y_0(N))\ontomap \rmH_1(X_0(N))).
\]
For $\omega_f,\omega_g\in\calB_h$ and $c\in \Cusp\setminus\{\infty\}$, let $e_{f,g,c}$ denote the image of \[
([\omega_f]\otimes[\overline{\omega_g}]-[\overline{\omega_f}]\otimes[\omega_g])\otimes[\eta_c]^\vee
\]
in the extension group $\Ext^1_{\mathcal{MH}_\bfR^+}(\bfR(0),H^1(X_0(N))^{\otimes 2}(2)\otimes V)$ (see Example~\ref{ex4-1}).
\begin{theorem}\label{thm5-1}
	If the base point $b$ lies over an $\bfR$-rational point of $Y_0(N)$, then the following identity holds$:$
	\begin{multline*}
		\label{eq37}
		\reg_2([I_b/I_b^3])=-2\pi\sqrt{-1}\sum_{\substack{f,g \in \calB_h \\ f \neq g},\  c\in\Cusp\setminus\{\infty\}}\frac{|G(\widehat{\bfZ}):K_\rmf|I_N(1;f,g;c)}{\zeta^{(N)}(2)(\omega_f,\omega_f)(\omega_g,\omega_g)}L^{(N)}(1,\pi_f\times\pi_g)e_{f,g,c}\\
		-2\pi\sqrt{-1}\sum_{f\in \calB_h,\ c\in\Cusp\setminus\{\infty\}}\frac{|G(\widehat{\bfZ}):K_\rmf|}{\zeta^{(N)}(2)(\omega_f,\omega_f)^2}\mathrm{Res}_{s=1}(L^{(N)}(s,\pi_f\times\pi_f))I_N(1;f,f;c)e_{f,f,c}\\
		-\sum_{f\in \calB_h,\ c\in\Cusp\setminus\{\infty\}}\frac{\calE_c(b,1)}{(\omega_f,\omega_f)}e_{f,f,c}.
	\end{multline*}
	If $b$ is the standard tangential base point $\partial/\partial q$ $($\cite[Section 4.1]{Brown17}$)$, then the following identity holds$:$
	\begin{multline*}
		\reg_2([I_b/I_b^3])=-2\pi\sqrt{-1}\sum_{f,g\in \calB_h,\ f\neq g,\ c\in\Cusp\setminus\{\infty\}}\frac{|G(\widehat{\bfZ}):K_\rmf|I_N(1;f,g;c)}{\zeta^{(N)}(2)(\omega_f,\omega_f)(\omega_g,\omega_g)}L^{(N)}(1,\pi_f\times\pi_g)e_{f,g,c}\\
		-2\pi\sqrt{-1}\sum_{f\in \calB_h,\ c\in\Cusp\setminus\{\infty\}}\frac{|G(\widehat{\bfZ}):K_\rmf|}{\zeta^{(N)}(2)(\omega_f,\omega_f)^2}\mathrm{Res}_{s=1}(L^{(N)}(s,\pi_f\times\pi_f))I_N(1;f,f;c)e_{f,f,c}.
	\end{multline*}
\end{theorem}
\begin{proof}
	This is a direct consequence of Theorem~\ref{thm4-2}, Proposition~\ref{prop5-1}, and Proposition~\ref{prop5-2}.
\end{proof}
Next, we compute $\reg_3([I_b/I_b^3])$.
For $f\in \calB_h$ and $ c,d\in \Cusp\setminus\{\ \infty\ \}$, define the element $e_{f,c,d}$ of $\Ext^1(\bfR(0),H^1(X_0(N))(2)\otimes V^{\otimes 2})$ to be the image of $([\omega_f]-[\overline{\omega_f}])\otimes [\eta_c]^\vee\otimes[\eta_d]^\vee$ (see Example~\ref{ex4-0}).
\begin{theorem}
	\label{thm5-2}The following identity holds$:$
	\begin{multline}
		\reg_3([I_b/I_b^3])=2\pi\sqrt{-1}\sum_{f\in\calB_h,  \substack{c,d\in\Cusp\setminus\{\infty\}\\
				c\neq d}}\left(\frac{|G(\widehat{\bfZ}):K_\rmf|(I_N(1;E_d,f;c)-I_N(1,E_c,f;d))}{\zeta^{(N)}(2)(\omega_f,\omega_f)}\right.\\
		\left.L^{(N)}\left(\frac{1}{2},\pi_f\right)L^{(N)}\left(\frac{3}{2},\pi_f\right)\right)e_{f,c,d}.
	\end{multline}
\end{theorem}
\begin{proof}
	Let $\calF_{\eta_c,\eta_d}$ be the smooth one-form on $Y_0(N)(\bfC)$ defined in (\ref{eq27}),
	namely,
	\[
	\calF_{\eta_c,\eta_d}:=\calE_c(\tau,1)\calR_-(\eta_d)-2\calE_d(\tau,1)\calR_-(\eta_c).
	\]
	Then, by Proposition~\ref{prop5-1} and Proposition~\ref{prop5-2}, we obtain the following identities:
	\begin{equation}
		\begin{split}
			(\calF_{c,d},\omega_f)&=(\calE_c\eta_d,\omega_f)-(\calE_d\eta_c,\omega_f)\\
			&=|G(\widehat{\bfZ}):K_{\rmf}|\zeta^{(N)}(2)^{-1}\left(L^{(N)}(1,\pi_{E_d}\times\pi_f)I_N(1;E_d,f;c)\right.\\
			&\hspace{0.5cm}\left.-L^{(N)}(1,\pi_{E_c}\times\pi_f)I_N(1;E_c,f;d)\right).
		\end{split}
	\end{equation}
	Since the Satake parameter of $\pi_{E_c,p}$ is $\{p^{1/2},p^{-1/2}\}$ for all $p\nmid N$ and $c\in \Cusp\setminus\{\infty\}$, we obtain the conclusion of the theorem by Theorem~\ref{thm4-3}.
\end{proof}

\subsection{On calculations of local zeta integrals}
The remaining task is to compute the local integrals $I_N(s;f,g;c)$.
We restrict ourselves to the special case where $f$ and $g$ are cuspidal newforms, and $c$ is a cusp near $0$ (see Definition~\ref{nearzero}). The same method applies to general cusps, but we do not pursue this here, as the argument is both lengthy and somewhat tangential to our main purpose.

For any prime number $p$, let $K_p$ denote the $p$-component of $K_{\rmf}$.
That is, we have
\[
K_p=\left\{\begin{bmatrix}
	a&b\\
	c&d
\end{bmatrix}\in G(\bfZ_p)\ \middle |\ c\in N\bfZ_p\right\}.
\]
The set $\Cusp$ decomposes as follows:
\[
\Cusp=\prod_{p}\Cusp_p,\quad \Cusp_p:=U(\bfZ_p)\backslash G(\bfZ_p)/K_p.
\]
For a cusp $c$, its $p$-component is denoted by $c_p$, which is an open compact subset of $G(\bfZ_p)$.

In the rest of this subsection, we fix normalized cuspidal newforms $f$ and $g$ of weight two and level $\Gamma_0(N)$.
For a prime number $p$, the additive character $\psi_p\colon \bfQ_p\to \bfC^\times$ is taken to be the same as in the previous subsection.
Since $f$ and $g$ are newforms, their Whittaker functions decomposes as pure tensors:
\[
W_f=\otimes'_vW_{f,v},\quad W_g=\otimes'_vW_{g,v},
\]
where $W_{f,p}$ and $W_{g,p}$ are normalized local newforms of $\pi_{f,p}$ and $\pi_{g,p}$,
respectively.
Let $\calK(\pi_{f,p},\psi_p)$ and $\calK(\pi_{g,p},\psi_p)$ denote the Kirillov model of $\pi_{f,p}$ and $\pi_{g,p}$ associated with $\psi_p$, respectively.
We define the elements $\xi_{f,p}\in \calK(\pi_{f,p},\psi_p)$ and $\xi_{g,p}\in \calK(\pi_{g,p},\psi_p)$ by
\[
\xi_{f,p}(x)=W_{f,p}\left(\begin{bmatrix}
	x&0\\
	0&1
\end{bmatrix}\right),\quad \xi_{g,p}(x)=W_{g,p}\left(\begin{bmatrix}
	x&0\\
	0&1
\end{bmatrix}\right).
\]

For a complex number $s$ and cusp $c$, we define
\begin{equation}\label{eq41}
	I(s,c_p):=\#(c_p/K_p)^{-1}\int_{\bfQ_p^\times \times c_p}(\pi_{f,p}(k)\xi_{f,p})(a)(\pi_{g,p}(k)\overline{\xi}_{g,p})(a)|a|_p^{s-1}d^\times ad k.
\end{equation}
When $c=\infty=K_{\rmf}$, we have the identity
\begin{equation}\label{eq38}
	I(s,\infty_p)=\mathrm{vol}(K_p)\int_{\bfQ_p^\times}\xi_{f,p}(a)\overline{\xi}_{g,p}(a)|a|_p^{s-1}d^\times a
\end{equation}
because $\xi_{f,p}$ and $\xi_{g,p}$ are invariant under the action of $K_p$.
\begin{remark}
	Our zeta integrals (\ref{eq41}), (\ref{eq37}) differ from those of Gelbart--Jacquet~\cite[(1.1.3)]{GJ}.
	Consequently, it is worth noting that the associated L-factors also differ from those in~\cite{GJ}.
	Compare~\cite[Proposition (1,4)]{GJ} with Proposition~\ref{prop6-1} below.\label{rem6.0}
\end{remark}
\begin{lemma}
	For each cusp $c$, the following identity holds$:$
	\[
	I_N(s;f,g;c)=\prod_{p|N}I(s,\infty_p)-\prod_{p|N}I(s,c_p).
	\]
\end{lemma}
\begin{proof}This is a direct consequence of the definition of $\phi_c$.
\end{proof}
From now on,
we fix a prime number $p$ dividing $N$, and write $\pi$ and $\pi'$ as $\pi_{f,p}$ and $\pi_{g,p}$, respectively.
Similarly, we write $\xi$ and $\xi'$ for $\xi_{f,p}$ and $\xi_{g,p}$, respectively.
Let $v_p\colon \bfQ_p^\times \to \bfZ$ be the additive valuation satisfying $v_p(p)=1$.
Note that the central characters of $\pi$ and $\pi'$ are trivial. Therefore, when $v_p(N)=1$, $\pi$ and $\pi'$ are special representations $\chi\mathrm{St}$ and $\chi'\mathrm{St}$,
respectively (\cite[Section 1.2]{Schmidt}). Here, $\chi$ and $\chi'$ are either the trivial character or the unramified quadratic character of $\bfQ_p^\times$.
The computation of $I(s,\infty_p)$ is straightforward.
\begin{proposition}\label{prop6-1}
	If $v_p(N)>1$, then the following identity holds:
	\[
	I(s,\infty_p)=\mathrm{vol}(K_p).
	\]
	If $v_p(N)=1$ and $\pi=\chi\mathrm{St}$,  $\pi'=\chi'\mathrm{St}$,
	then the following identity holds:
	\[
	I(s,\infty_p)=\mathrm{vol}(K_p)\frac{1}{1-\chi(p)\chi'(p)p^{-s-1}}.
	\]
\end{proposition}
\begin{proof}
	By definition, $\xi$ and $\xi'$ are normalized local newforms.
	Therefore, by~\cite[Summary of Section 2]{Schmidt}, both $\xi$ and $\xi'$ are the characteristic function of $\bfZ_p^\times$ when $v_p(N)>1$.
	Thus, the conclusion follows directly from (\ref{eq38}).
	
	When $v_p(N)=1$, we have
	\[
	\xi(x)=|x|_p\chi(x)\mathrm{char}_{\bfZ_p}(x),\quad \xi'(x)=|x|_p\chi'(x)\mathrm{char}_{\bfZ_p}(x)
	\]
	from the same table of local newforms. Therefore, the conclusion also follows by direct computation.
\end{proof}
Define the matrix $w_p\in G(\bfZ_p)$ by
\[
w_p=\begin{bmatrix}
	0&1\\
	-1&0
\end{bmatrix}
\]
and define $0_p$ to be the double coset $U({\bfZ_p})w_pK_p$.
In this paper, we compute only the local integral $I(s,0_p)$.
\begin{lemma}For $a\in \bfQ_p$, define $\tau_a$ to be $\begin{bmatrix}
		1&a\\
		0&1
	\end{bmatrix}$.
	Put $r:=v_p(N)$.
	Then we have 
	\[
	0_p=\coprod_{a=0}^{p^r-1}\tau_aw_pK_p.
	\]\label{lem5-1}
\end{lemma}
\begin{proof}
	Since $w_p^{-1}\tau_aw_p=\begin{bmatrix}
		1&0\\
		-a&1
	\end{bmatrix}$, we have $\tau_a w_p K_p=\tau_b w_pK_p$ if and only if $a-b\in p^r\bfZ_p$.
	Therefore, the cosets $\{\tau_aw_pK_p\}_{a=0}^{p^r-1}$ cover $0_p$, and they are pairwise disjoint.
\end{proof}
The epsilon factors of $\pi$ and $\pi'$, associated with an additive character $\psi$, are denoted by $\epsilon(s,\pi,\psi)$ and $\epsilon(s,\pi',\psi)$,
respectively. According to~\cite[(10)]{Schmidt}, their values at $s=1/2$ do not depend on the choice of $\psi$. Therefore, we write $\epsilon(1/2,\pi)$ and $\epsilon(1/2,\pi')$ for their values at $s=1/2$, respectively.
\begin{proposition}
	Let $r:=v_p(N)$. Then, the following identity holds$:$
	\[
	I(s,0_p)=\epsilon\left(\frac{1}{2},\pi\right)\epsilon\left(\frac{1}{2},\pi'\right)p^{r(s-1)}I(s,\infty_p).
	\]
\end{proposition}
\begin{proof}Suppose that $r>1$. According to~\cite[Proof of Theorem 3.2, Case $n\geq 2$]{Schmidt}, we have
	\[
	\pi(w_p)\xi(x)=\epsilon\left(\frac{1}{2},\pi\right)\mathrm{char}_{p^{-r}\bfZ_p^\times}(x),\quad \pi(w_p)\xi'(x)=\epsilon\left(\frac{1}{2},\pi'\right)\mathrm{char}_{p^{-r}\bfZ_p^\times}(x).
	\]
	Therefore, by Lemma~\ref{lem5-1}, the following identities hold:
	\begin{equation*}
		\begin{split}
			I(s,0_p)&=\frac{\mathrm{vol}(K_p)}{p^r}\sum_{a=0}^{p^r-1}\int_{\bfQ_p^\times }(\pi(\tau_a w_p)\xi)(x)\overline{(\pi'(\tau_a w_p)\xi')(x)}|x|^{s-1}d^\times x\\
			&=\mathrm{vol}(K_p)\int_{\bfQ_p^\times }(\pi( w_p)\xi)(x)\overline{(\pi'( w_p)\xi')(x)}|x|^{s-1}d^\times x\\
			&=\mathrm{vol}(K_p)\epsilon\left(\frac{1}{2},\pi\right)\epsilon\left(\frac{1}{2},\pi'\right)p^{r(s-1)}.
		\end{split}
	\end{equation*}
	
	If $r=1$ and $\pi=\chi\mathrm{St},\ \pi'=\chi'\mathrm{St}$, then the identities
	\[
	\pi(w_p)\xi(x)=\epsilon\left(\frac{1}{2},\pi\right)\xi(px),\quad \pi(w_p)\xi'(x)=\epsilon\left(\frac{1}{2},\pi'\right)\xi'(px)
	\]
	hold by ~\cite[Proof of Theorem 3.2, Case $n=1$]{Schmidt}. Therefore, we have
	\begin{equation*}
		\begin{split}
			I(s,0_p)&=\mathrm{vol}(K_p)\int_{\bfQ_p^\times }(\pi( w_p)\xi)(x)\overline{(\pi'( w_p)\xi')(x)}|x|^{s-1}d^\times x\\
			&=\mathrm{vol}(K_p)\epsilon\left(\frac{1}{2},\pi\right)\epsilon\left(\frac{1}{2},\pi'\right)\int_{\bfQ_p^\times }\xi(px)\overline{\xi'(px)}|x|^{s-1}d^\times x\\
			&=\mathrm{vol}(K_p)\epsilon\left(\frac{1}{2},\pi\right)\epsilon\left(\frac{1}{2},\pi'\right)p^{s-1}\int_{\bfQ_p^\times }\xi(x)\overline{\xi'(x)}|x|^{s-1}d^\times x.
		\end{split}
	\end{equation*}
	Thus, we obtain the conclusion of the proposition.
\end{proof}
\begin{definition}\label{nearzero}
	Let $c$ be a cusp different from $\infty$. We say that $c$ is \emph{near $0$} if $c_p$ is $0_p$ or $\infty_p$
	for all prime numbers $p$.
\end{definition}
For simplicity of notation, for $\pi=\pi_{f,p}$ and $\pi'=\pi_{g,p}$, we set
\begin{equation}\label{eq39}
	L_p(s,\pi\times \pi'):=I(s,\infty_p)\mathrm{vol}(K_p)^{-1}
\end{equation}
and define
\begin{equation}
	\label{eq40}
	L(s,\pi_f\times \pi_g):=\prod_{p\text{ prime numbers}}L_p(s,\pi_{f,p}\times \pi_{g,p}).
\end{equation}
\begin{theorem}\label{thm5-3}
	Let $c$ be a cusp near $0$. Then we have
	\[
	I(s;f,g;c)=L(s,\pi_f\times \pi_g)|G(\widehat{\bfZ}):K_{\rmf}|^{-1}\left(1-\prod_{p|N,\ c_p=0_p}\epsilon\left(\frac{1}{2},\pi_{f,p}\right)\epsilon\left(\frac{1}{2},\pi'_{g,p}\right)|N|_p^{1-s}\right)
	\]
\end{theorem}
\begin{proof}
	This is a direct consequence of Proposition~\ref{prop5-1} and Proposition~\ref{prop5-2}.
\end{proof}
When $N$ is square-free, every cusp is near to $0$. Therefore,
we can compute the ``newform part'' of $\reg_2([I_b/I_b^3])$ explicitly by the theorem above.
The explicit computation in the case $N=p$ is exactly Theorem~\ref{thm1} in the Introduction.
Note that in this case, the space of Eisenstein series is one-dimensional, and $\reg_3([I_b/I_b^3])$ consequently vanishes.
%%%%%%%%%%%%%%%%%%%%%%%%
%%%%%%%%%%%%%%%%%%%%%%%%
%%%%%%%%%%%%%%%%%%%%%%%%
\appendix
\section{Universal objects of unipotent Tannakian categories}\label{unittan}
In this appendix, we summarize basic properties of unipotent Tannakian categories.
\begin{comment}
It seems that the results in this section is very well-known facts for professional readers
(see~\cite{Kim} for example).
However, the author do not find a suitable explicit reference about it and decided to add this appendix for readers' convenience.
\end{comment}

Let $\calK$ be a field of characteristic zero,
and let $\calC$ be a $\calK$-linear neutral Tannakian category (~\cite{DeM}).
The symbol $\boldsymbol 1$ denotes the unit object of $\calC$.
We say that an object $V$ of $\calC$ is \emph{unipotent} if it is isomorphic to an iterated extension of $\boldsymbol 1$.
That is, $V$ is unipotent if there exists a sequence
\[
0=V_0\subset V_1\subset\cdots \subset V_n=V
\]
of objects of $\calC$ such that $V_{i+1}/V_i$ is isomorphic to a finite direct sum of $\boldsymbol 1$.
The minimum length of such a sequence is called the \emph{unipotency of $V$}.
When every object of $\calC$ is unipotent, we say that $\calC$ is a \emph{unipotent Tannakian category}.
For a unipotent Tannakian category $\calC$ and a positive integer $N$,
$\calC^{\leq N}$ is defined to be the strictly full subcategory of $\calC$
consisting of all objects whose unipotency are less than or equal to $N$.

From now on, suppose that $\calC$ is a unipotent Tannakian category.
\begin{proposition}
	Suppose that the $\calK$-vector space $\Ext^1_{\calC}(\boldsymbol 1,\boldsymbol 1)$ is finite-dimensional. Put $V_1:=\boldsymbol 1$.
	Then there exists an inductive system $\{V_N\}_{N\geq 1}$ in $\calC$, equipped with exact sequences
	\begin{equation}
		\label{eqa1}
		0\to V_{N-1}\to V_N\to \Ext^1_{\calC}(\boldsymbol 1,V_{N-1})\otimes \boldsymbol 1\to 0
	\end{equation}
	for all $N\geq 2$,
	such that the connecting homomorphism
	\[
	\Ext^1_{\calC}(\boldsymbol 1,V_{N-1})\to \Ext^1_{\calC}(\boldsymbol 1,V_{N-1})
	\]
	induced by $\Hom(\boldsymbol 1,-)$ is the identity map.\label{propa1}
\end{proposition}
\begin{proof}
	Note that we have a canonical isomorphism
	\[
	\Ext^1_{\calC}(\Ext^1_{\calC}(\boldsymbol 1,V_{N-1})\otimes\boldsymbol 1,V_{N-1})\cong \End(\Ext^1_{\calC}(\boldsymbol 1,V_{N-1})).
	\]
	Then, we take $V_N$ to be a representative of the extension class corresponding to the identity map. Note that, by induction on $N$, we can show that $\Ext^1_{\calC}(\boldsymbol 1,V_N)$ is
 finite-dimensional for all $N$. Therefore, such a $V_N$ exists in $\calC$ for all $N$.
\end{proof}
In the rest of this appendix, we suppose that the assumption in Proposition~\ref{propa1} holds.
By construction, any subextension of (\ref{eqa1}) does not split.
Indeed, if \[
0\to V'\to V\to \boldsymbol 1\to 0
\]
is such a split extension such that $V'\subset V_{N-1}$,
then pushout of this extension by the inclusion $V'\subset V_{N-1}$, which is a subobject of $V_{N-1}$,
would also split, contradicting the property of (\ref{eqa1}).
Thus, we have the following lemma.
\begin{lemma}
	\label{lema1}Let $N$ be a positive integer and let $V$ be a subobject of $V_{N}$.
	If the unipotency of $V$ is less than $N$, then $V$ is a subobject of $V_{N-1}$.
\end{lemma}
\begin{lemma}
	\label{lema2}Let $V$ be an object of $\calC^{\leq N}$, and let
	\[
	0\to V'\to V\to \boldsymbol 1^{\oplus r}\to 0
	\]
	be an exact sequence with $V'\in \Obj(\calC^{\leq N-1})$.
	Then, the natural homomorphism
	\[
	\Hom_{\calC}(V,V_N)\to \Hom_{\calC}(V',V_N)
	\]
	is surjective.
\end{lemma}
\begin{proof}Let $f\colon V'\to V_N$ be a morphism.
	Then, by Lemma~\ref{lema1}, the image of $f$ is contained in $V_{N-1}$.
	Let
	\[
	0\to V_{N-1}\to \widetilde V\to \boldsymbol 1^{\oplus r}\to 0
	\]
	be the pushout of the sequence in the lemma by $f$.
	Each $i$-th component $\boldsymbol 1\to \boldsymbol 1^{\oplus r}$ defines an element $s_i$ of $\Ext^1_{\calC}(\boldsymbol 1,V_{N-1})$.
	By the definition of (\ref{eqa1}), there exists the dotted allow
	\[
	\xymatrix{
		0\ar[r]& V_{N-1}\ar[r]\ar[d]^{\mathrm{id}}& \widetilde V\ar[r]\ar@{-->}[d]& \boldsymbol 1^{\oplus r}\ar[r]\ar[d]^{\sum s_i}& 0\\
		0\ar[r]&V_{N-1}\ar[r]&V_N\ar[r]&\Ext^1_{\calC}(\boldsymbol 1,V_{N-1})\otimes\boldsymbol 1\ar[r]&0,
	}
	\]
	which makes the diagram commute.
	Then, the composition of $V\to \widetilde V$ with this dotted arrow defines an extension of $f$ to $V_N$.
	This completes the proof of the lemma.
\end{proof}
\begin{proposition}
	\label{propa2}Let $\calC^{\op}$ denote the opposite category of $\calC$. Then, the functor
	\[
	\Hom_{\calC}(-,V_\infty)\colon \calC^{\op}\to \Vec_{\calK};\quad V\mapsto \varinjlim_N\Hom_{\calC}(V,V_N)
	\]
	is a $\calK$-linear exact functor.
\end{proposition}
\begin{proof}
	Since this functor is clearly left exact, it suffices to show the surjectivity of the map
	$\Hom_{\calC}(V,V_\infty)\to \Hom_{\calC}(V',V_\infty)$ for any injection $V'\intomap V$.
	
	Let $N$ be the unipotency of $V$.
	Then, according to Lemma~\ref{lema1}, it suffices to show that the map
	$\Hom_{\calC}(V,V_N)\to \Hom_{\calC}(V',V_N)$ is surjective.
	We prove this surjectivity by the induction on $N$.
	
	Suppose that the assertion holds for all objects in $\calC^{\leq N-1}$.
	Let $V_0$ be a subobject of $V$ such that $V_0\in\Obj(\calC^{\leq N-1})$ and that $V/V_0$ is isomorphic
	to a direct sum of copies of $\boldsymbol 1$.
	Let $V_0'$ denote the pull-back of $V_0$ to $V'$.
	Then, by Lemma~\ref{lema2}, we obtain the commutative diagram
	\[
	\xymatrix{
		0\ar[r]&\Hom_{\calC}(V/V_0,V_N)\ar[r]\ar[d]&\Hom_{\calC}(V,V_N)\ar[r]\ar[d]&\Hom_{\calC}(V_0,V_N)\ar[r]\ar[d]&0\\
		0\ar[r]&\Hom_{\calC}(V'/V'_0,V_N)\ar[r]&\Hom_{\calC}(V',V_N)\ar[r]&\Hom_{\calC}(V'_0,V_N)\ar[r]&0
	}
	\]
	with exact rows.
	Then, the left and the right vertical homomorphisms are surjective by the induction hypothesis.
	Therefore, the middle vertical map is also surjective by the snake lemma. This completes the proof.
\end{proof}

Let \[
\omega\colon \calC\to \Vec_{\calK}
\]
be a fiber functor on $\calC$. An $\omega$-comarked object means a pair
$(V,s)$ where $V$ is an object of $\calC$ and $s$ is a $\calK$-linear map
$\omega(V)\to \calK$. Clearly, $s$ can be regarded as an element of $\omega(V^\vee)$, where $V^\vee$ is
the dual object of $V$.
Let $\{V_N\}_N$ be the same as in the proposition above.
Then, we can easily extend this system to an inductive system $\{(V_N,s_N)\}_{N\geq 1}$ of
$\omega$-comarked objects such that $s_1\colon \omega(\boldsymbol 1)\cong \calK\to \calK$
is the identity map.
\begin{theorem}For any object $V$ of $\calC^{\leq N}$, the natural $\calK$-linear map
	\[
	\Hom_{\calC}(V,V_N)\to \Hom_{\calK}(\omega(V),\calK);\quad f\mapsto s_N\circ \omega(f)
	\]
	is bijective.
\end{theorem}
\begin{proof}We prove this assertion by induction on $N$.
	Suppose that the assertion holds for $N-1$.
	Let
	\[
	0\to V'\to V\to \boldsymbol 1^{\oplus r}\to 0
	\]
	be an exact sequence such that $V_1\in \Obj(\calC^{\leq N-1})$.
	According to Proposition~\ref{propa2}, we have the following commutative diagram
	\[
	\xymatrix{
		0\ar[r]& \Hom(\boldsymbol 1^{\oplus r},V_N)\ar[r]\ar[d]& \Hom(V,V_N)\ar[d]\ar[r]& \Hom(V',V_N)\ar[d]\ar[r]&0\\
		0\ar[r]&\omega^\vee(\boldsymbol 1^{\oplus r})\ar[r]&\omega^\vee(V)\ar[r]&\omega^\vee(V')\ar[r]&0
	}
	\]with exact rows.
	Then, the left and right vertical homomorphisms are isomorphisms
	by the induction hypothesis and Lemma~\ref{lema2}.
	Therefore, we obtain the conclusion of the theorem by the snake lemma.
\end{proof}
The following corollaries are direct consequences of the theorem above.
\begin{corollary} Define $\omega^\vee\colon \calC^{\op}\to \Vec_{\calK}$ to be $\Hom_{\Vec_{\calK}}(\omega(-),\calK)$,
	where $\calC^{\op}$ is the opposite category of $\calC$.
	Then, the $\omega$-comarked object $(V_N,s_N)$ represents the functor $\omega^\vee|_{\calC^{\leq N}}$, and the ind-$\omega$-comarked object $(V_\infty,s_\infty)=\varinjlim_N(V_N,s_N)$
	represents the functor $\omega^\vee$.
\end{corollary}
\begin{corollary}
	\label{cora1}The ind-$\omega$-comarked object $(V_\infty,s_\infty)$ is unique up to a unique isomorphism.
\end{corollary}
\begin{corollary}
	\label{cora2}The functor $\Hom_{\calC}(-,V_\infty)$ is a fiber functor of $\calC^{\mathrm{op}}$.
\end{corollary}
\begin{comment}
\begin{corollary}
	\label{cora3}For an arbitrary fiber functor $\omega_1\colon \calC\to \Vec_{\calK}$,
	let $\Hom(\omega_1,\omega)$ denote the set of homomorphisms of functors.
	Then, the natural homomorphism
	\[
	\Hom(\omega_1,\omega)\to \omega_1(V_\infty)^\vee;\quad f\mapsto s_\infty\circ f_{V_\infty}
	\]
	is an isomorphism.
\end{corollary}
\end{comment}
\begin{definition}
	\label{dfna1}Let $\calC$ be a unipotent Tannakian category satisfying the condition of Proposition~\ref{propa1}. Then, $(V_\infty,s_\infty)$ is called the \emph{universal ind-$\omega$-comarked object} of $\calC$. Its \emph{$N$th layer} refers $(V_N,s_N)$.
	
	Dually, the pro-marked object $(V_\infty^\vee,s_\infty)=\varprojlim_N(V_N^\vee,s_N)$
	is called the \emph{universal pro-$\omega$-marked object} of $\calC$.
	Its \emph{$N$th layer} refers the marked object $(V_N^\vee,s_N)$ of $\calC$.
\end{definition}


\begin{thebibliography}{99}
	\bibitem{Beilinson}A. Beilinson, Higher regulators and values of L-functions, \textit{J. Soviet Math.} \textbf{30} (1985), 2036--2070.
	
	\bibitem{Brown17}F. Brown, Multiple modular values and relative completion of the fundamental group of $\mathcal M_{1,1}$, preprint, arXiv:1407.5167.
	
%	\bibitem{BC}F. Brunault and M. Chida, Regulators for Rankin-Selberg products of modular forms, \textit{Ann. Math. Qu\'e.} \textbf{40} (2016), no. 2, 221--249.
	
	\bibitem{Bump}D. Bump, Automorphic forms and representations, Cambridge Studies in Advanced Mathematics \textbf{55} (Cambridge University Press, 1997).
	
	%	\bibitem{Cushman}M. Cushman, The motivic fundamental group, PhD thesis, Department of Mathematics, University of Chicago, (2000).
	
	\bibitem{DRS15}H. Darmon, V. Rotger, and I. Sols, Iterated integrals, diagonal cycles and rational points on elliptic curves (English, with English and French summaries), Publications math\'ematiques de Besan\c{c}on. Alg\'ebre et th\'eorie des nombres, 2012/2, \textit{Publ. Math. Besan\c{c}on Alg\'ebre Th\'eorie}, vol. 2012/, Presses Univ. Franche-Comt\'e, Besan\c{c}on, 2012, pp. 19--46.
	
%	\bibitem{DN}F. D\'eglise, W. Nizio\l, \textit{On $p$-adic absolute Hodge cohomology and syntomic coefficients, I}, Comment.\ Math.\ Helv., \textbf{ 93} (2018), no.\ 1, 71--131.
	
	\bibitem{Del89}P. Deligne, Le groupe fondamental de la droite projective moins trois points, Galois groups over $\mathbf Q$ (Berkeley, CA, 1987), 79--297, \textit{ Math. Sci. Res. Inst. Publ.}, \textbf{ 16}, Springer, New York, 1989. 
	
	\bibitem{Del97}P. Deligne, Local behavior of Hodge structures at infinity, Mirror Symmetry II (Green, Yau eds.) AMS/IP Stud. Adv. Math. \textbf{1}, 683--699, Amer. Math. Soc. 1993.
	
	\bibitem{De-G05}P. Deligne, A. Goncharov, Groupes fondamentaux motiviques de Tate mixte, \textit{ Ann. Sci. Ecole Norm. Sup\'er.} (4) \textbf{ 38} (2005), 1--56.
	\bibitem{DeM}P. Deligne, J. Milne, Tannakian categories, in Hodge Cycles, Motives and Shimura varieties, Lecture Notes in Math., \textbf{900}, Springer-Verlag, Berlin, 1982.
	
	\bibitem{DS}F. Diamond, J. Shurman, A First Course in Modular Forms, Graduate Texts in Mathematics, vol. \textbf{228}. Springer, New York (2005).
	
	\bibitem{GJ}S. Gelbart, H. Jacquet, \textit{A relation between automorphic representations of $\mathrm{GL}(2)$ and $\mathrm{GL}(3)$}, \textit{Ann. Sci. Ecole Norm. Sup\'er.}, \textbf{11} (1978), 471--542.
	
	\bibitem{Hain87}R. Hain, \textit{The geometry of the mixed Hodge structures on the fundamental group}, Proc. Symp. Pure Math., \textbf{46} (1987), 247--282.
	
	\bibitem{Hain87b}R. Hain, The de Rham homotopy theory of complex algebraic varieties I, \textit{K-theory} \textbf{1} (1987), 271--324.
	
	\bibitem{Hain15}R. Hain, The Hodge-de Rham theory of modular groups, in Recent Advances in Hodge Theory, London Mathematical Society Lecture Note Series, \textbf{427} (Cambridge University Press, Cambridge, 2016), 422--514.
	
%	\bibitem{HM20}R.\ Hain and M.\ Matsumoto, Universal mixed elliptic motives,	{\it J. Inst. of Math. Jussieu}, \textbf{19} no. 3 (2020) 663--766.
	
	\bibitem{HainZucker}R. Hain, S. Zucker, Unipotent variations of mixed Hodge structure, \textit{Invent. Math.} \textbf{88} (1987) 83--124. 
	
	\bibitem{HuM17}A. Huber, U. M\"ullerstach, Periods and Nori motives, Ergebnisse der Mathematik und ihrer Grenzgebiete, 3, Folge., vol. \textbf{65}, Springer, Cham, 2017, With contributions by Benjamin Friedrich and Jonas von Wangenheim.
	
	\bibitem{Ihara86}Y. Ihara, Profinite braid groups, Galois representations, and complex multiplications, \textit{Ann. of Math.} \textbf{123} (1986), 43--106.
	
	\bibitem{Ja}H. Jacquet, Automorphic Forms on $\mathrm{GL}(2)$: Part II, Springer Lecture Notes, \textbf{278}, 1972.
	
%	\bibitem{JL}H. Jacquet, R. P. Langlands, Automorphic forms on GL(2), Lecture Notes in Mathematics,
%	Vol. \textbf{114}. Springer-Verlag, Berlin, 1970.
	
	\bibitem{KM84}N. Katz, B Mazur, Arithmetic moduli of elliptic curves, Princeton Univ. Press, Princeton, 1984.
	%Annals of Math. Studies 108, Princeton Univ. Press.
	
	\bibitem{Kim}M. Kim, The unipotent Albanese map and Selmer varieties for curves, \textit{Publ. RIMS} \textbf{45} (2009) 89--133.
	
%	\bibitem{KLZ17}G. Kings, D. Loeffler, and S. L. Zerbes, Rankin--Eisenstein classes and explicit reciprocity laws, \textit{Cambridge J. Math.} \textbf{5} (2017), no. 1, 1--122.
	
	\bibitem{Luo}M. Luo, Mixed Hodge structures on fundamental groups, available at \url{http://people.maths.ox.ac.uk/luom/notes/mhsfg.pdf}.
	
	\bibitem{Morgan}J. W. Morgan, The algebraic topology of smooth algebraic varieties, \textit{Inst. Hautes \'Etudes Sci. Publ. Math.} No. \textbf{48} (1978), 137--204.
	
	\bibitem{Miyake}T. Miyake, Modular forms, Springer-Verlag, Berlin, 1989, translated from the Japanese by Yoshitaka Maeda. 
	
	\bibitem{Nekovar}J. Nekov\'a\u{r}, Beilinson's conjectures, Motives (Seattle, WA, 1991), 537--570, Proc. Sympos. Pure
	Math., 55, Part 1, Amer. Math. Soc., Providence, RI, (1994).
	
	\bibitem{PP}V. Pasol, A. Popa, Modular forms and period polynomials, \textit{Proc. Lond. Math. Soc.} (3) \textbf{107} (2013), no. 4, 713--743.
	
	\bibitem{SerreLie}J.-P. Serre, Lie algebras and Lie groups, 1964 lectures given at Harvard University, Second edition, \textit{ Lecture Notes in Mathematics}, \textbf{1500}. Springer-Verlag, Berlin, 1992.
	
	\bibitem{SS}N. Schappacher and A. J. Scholl, Beilinson’s Theorem on Modular Curves, Beilinson’s Conjectures on Special Values of L-functions (M. Rapoport, N. Schappacher, P. Schneider, eds.),
	Perspectives in Mathematics, vol. 4, Academic Press, Boston, 1988, pp. 273--304.
	
	\bibitem{Shimura1}G. Shimura, Introduction to the Arithmetic Theory of Automorphic Functions, 1971.
	
	%	\bibitem{Shimura2}G. Shimura, The special values of the zeta functions associated with cusp forms, \textit{Comm. pure and appl. math.} \textbf{29}, 783--804 (1976).
	
%	\bibitem{Shimura}G. Shimura, On the periods of modular forms, \textit{Math. Annalen} \textbf{229} (1977), 211--221.
	
	\bibitem{Schmidt}R. Schmidt, Some remarks on local newforms for $\mathrm{GL}(2)$,
	\textit{J. Ramanujan Math.\ Soc.}, \textbf{17} (2002) 115--147.
	
	\bibitem{Voison} C. Voisin, Hodge Theory and Complex Algebraic Geometry. I, Translated from the French Original by Leila Schneps, Cambridge Studies in Advanced Mathematics, vol. \textbf{76}, Cambridge University Press, Cambridge, 2002.
	
	\bibitem{Wojtkowiak}Z. Wojtkowiak, Cosimplicial objects in algebraic geometry, in Algebraic Ktheory and Algebraic Topology, Kluver Academic Publishers, 1993, pp.\ 287--327.
	
	\bibitem{Zagier}D. Zagier, Polylogarithms, Dedekind zeta functions and the algebraic $K$theory of fields, Arithmetic algebraic geometry (Texel, 1989), 391--430, {\em Progr.~Math.}, \textbf{89}, Birkhauser Boston, Boston, MA, 1991. 
\end{thebibliography}
\end{document}